\documentclass[12pt]{amsart}

\usepackage{pinlabel}
\usepackage{mathrsfs,verbatim}
\usepackage[super]{nth}
\usepackage[colorinlistoftodos,color=red!70]{todonotes}
\usepackage[margin=1in]{geometry}
\usepackage{amsmath,amssymb,url}

\title[Analysis on Surreal Numbers]{Analysis on Surreal Numbers}

\author{Simon Rubinstein-Salzedo}

\address{Department of Statistics, Stanford University, 390 Serra Mall, Stanford, CA 94305, USA}
\email{simonr@stanford.edu}

\author{Ashvin Swaminathan}
\address{Department of Mathematics, Harvard College, 1 Oxford Street, Cambridge, MA 02138, USA}
\email{aaswaminathan@college.harvard.edu}


\newtheorem{theorem}{Theorem}
\newtheorem{proposition}[theorem]{Proposition}
\newtheorem{lemma}[theorem]{Lemma}

\newtheorem{conjecture}[theorem]{Conjecture}

\theoremstyle{definition}
\newtheorem{defn}[theorem]{Definition}
\newtheorem{example}[theorem]{Example}

\theoremstyle{remark}
\newtheorem*{remark}{Remark}

\makeatletter
\newcommand*{\defeq}{\mathrel{\rlap{%
                     \raisebox{0.3ex}{$\m@th\cdot$}}%
                     \raisebox{-0.3ex}{$\m@th\cdot$}}%
                     =}
\makeatother

\newcommand{\beas}{\begin{eqnarray*}}
\newcommand{\eeas}{\end{eqnarray*}}

\newcommand{\bm}[1]{{\mbox{\boldmath $#1$}}}

\newcommand{\wt}{\widetilde}

\newcommand{\nc}{\newcommand}

\nc{\nlog}{\mathrm{nlog}}
\nc{\ul}{\underline}
\nc{\ol}{\overline}
\nc{\BR}{\mathbb{R}}
\nc{\BZ}{\mathbb{Z}}

\begin{document}

\begin{abstract}
\noindent The class $\mathbf{No}$ of surreal numbers, which John Conway discovered while studying combinatorial games, possesses a rich numerical structure and shares many arithmetic and algebraic properties with the real numbers. Some work has also been done to develop analysis on $\mathbf{No}$. In this paper, we extend this work with a treatment of functions, limits, derivatives, power series, and integrals.

We propose surreal definitions of the arctangent and logarithm functions using truncations of Maclaurin series. Using a new representation of surreals, we present a formula for the limit of a sequence, and we use this formula to provide a complete characterization of convergent sequences and to evaluate certain series and infinite Riemann sums via extrapolation. A similar formula allows us to evaluates limits (and hence derivatives) of functions.

By defining a new topology on $\mathbf{No}$, we obtain the Intermediate Value Theorem even though $\mathbf{No}$ is not Cauchy complete, and we prove that the Fundamental Theorem of Calculus would hold for surreals if a consistent definition of integration exists. Extending our study to defining other analytic functions, evaluating power series in generality, finding a consistent definition of integration, proving Stokes' Theorem to generalize surreal integration, and studying differential equations remains open.
\end{abstract}


\maketitle

\section{Introduction}\label{intro}

\noindent Since their invention by John Conway in 1972, surreal numbers have intrigued mathematicians who wanted to investigate the behavior of a new number system. Even though surreal numbers were constructed out of attempts to describe the endgames of two-player combinatorial games like Go and Chess, they form a number system in their own right and have many properties in common with real numbers. Conway demonstrated in his book~\cite{Con01} that out of a small collection of definitions, numerous arithmetic and algebraic similarities could be found between reals and surreals. Using a creation process, starting with the oldest number (called ``$0$'') and progressing toward more nontrivial numbers, Conway proved that the surreals contain both the reals and the ordinals. After defining basic arithmetic operations (comparison, negation, addition, and multiplication) for surreals, he showed that the surreals contain never-before-seen numbers, such as $\omega^5-\left(\omega+3\pi\right)^2 \times \omega^{-\omega}$, that arise out of combining reals and ordinals ($\omega$ is the first transfinite ordinal).
By determining the properties of surreal arithmetic operations, Conway studied the algebraic structure of surreals, concluding that the surreals form an object (called ``$\mathbf{No}$'') that shares all properties with a totally ordered field, except that its elements form a proper class (in this regard, we write that $\mathbf{No}$ is a ``Field''). With surreal arithmetic and algebra in place, developing analysis on $\mathbf{No}$ is the next step in building the theory of surreal numbers. Below, we discuss earlier work on surreal analysis and introduce our own results.

The study of surreal functions began with polynomials, which were constructed using the basic arithmetic operations that Conway introduced in his book~\cite{Con01}. Subsequently, Gonshor found a definition of $\exp(x)$ that satisfies such fundamental properties as $\exp(x+y)=\exp(x)\cdot\exp(y)$ for all $x, y \in \mathbf{No}$ \cite{Gon86}. Moreover, Kruskal defined $1/x$, and Bach defined $\sqrt{x}$ \cite{Con01}. In this paper, we present a more rigorous method of constructing functions inductively, and we show that Gonshor's method for defining the exponential function can be utilized to define arctangent and logarithm, as is independently observed by Costin in~\cite{Cos12}.

Conway and Norton initiated the study of surreal integration by introducing a preliminary analogue of Riemann integration on surreals, as described in~\cite{Con01}. The ``Conway-Norton'' integral failed to have standard properties of real integration, however, such as translation invariance: $\int_{a}^{b} f(x) dx = \int_{a-t}^{b-t} f(x+t) dx$, for any surreal function $f(x)$ and $a,b,t \in \mathbf{No}$. While Fornasiero fixed this issue in~\cite{For04}, the new integral, like its predecessor, yields $\exp(\omega)$ instead of the desired $\exp(\omega) -1$ for $\int_{0}^{\omega}\exp(x) dx$.

One way of approaching the problem of integration on surreals is to give meaning to infinitely long ``Riemann'' sums.\footnote{Because the surreals contain the ordinals, ``Riemann'' sums of infinite length are considered, unlike in real analysis, where the term ``Riemann'' requires sums of finite length.} To do this, we need to know how to evaluate limits of surreal sequences. In the literature, surreal sequences are restricted to a limit-ordinal number of terms, but such sequences are not convergent in the classical $\varepsilon$-$\delta$ sense. Moreover, earlier work has not defined the limit, and hence the derivative, of a surreal function. In this paper, using a new representation of surreals, we obtain a formula for the limit of a surreal sequence. Although we show that $\mathbf{No}$ is not Cauchy complete, we use this formula to completely characterize convergent sequences. Our formula for the limit of a sequence also gives us a method of evaluating certain series and infinitely long ``Riemann'' sums by extrapolation from the naturals to the ordinals, and we show that this extrapolative method can correctly integrate the exponential function. We also present a method of finding limits (and hence derivatives) of surreal functions.

The difficulty of creating a definition of integration for surreals is attributed in~\cite{Cos10} to the fact that the topological space of surreals is totally disconnected when the standard notion of local openness is used; i.e.~given any locally open interval $(a,b) \in \mathbf{No}$, we have $(a,b) = (a,g) \cup (g,b)$, where $g$ is a gap between $a$ and $b$ and the intervals $(a,g)$ and $(g,b)$ are locally open. In this paper, to help deal with this difficulty, we define a new topology on $\mathbf{No}$, in which $\mathbf{No}$ is connected. Using this topology, we prove that the Intermediate Value Theorem holds even though $\mathbf{No}$ is not Cauchy complete, and we also prove that the Extreme Value Theorem holds for certain continuous functions. We then show that the Fundamental Theorem of Calculus would hold for surreals if we have a definition of integration that satisfies certain necessary properties.

The rest of this paper is organized as follows. Section~\ref{defs} discusses definitions and basic properties of surreals as well as our new topology on $\mathbf{No}$, and Section~\ref{newfunc} introduces our definitions for the arctangent and logarithm functions. Section~\ref{seqs} explains our method of evaluating limits of sequences and discusses the characteristics of convergent sequences. Section~\ref{diff} discusses limits of functions and includes the Intermediate Value Theorem, and Section~\ref{ips} concerns series, Riemann sums, and the Fundamental Theorem of Calculus. Finally, Section~\ref{conc} concludes the paper with a discussion of open problems.

\section{Definitions and Basic Properties}\label{defs}
\noindent In this section, we review all basic definitions and properties of surreals and introduce our own definitions and conventions. Throughout the rest of the paper, unqualified terms such as ``number,'' ``sequence,'' and ``function'' refer to surreal objects only. Any reference to real objects will include the descriptor ``real'' to avoid ambiguity. For an easy introduction to surreal numbers, see Knuth's book~\cite{Knu74}.

\subsection{Numbers} \label{defsnums}
\noindent Conway constructed numbers recursively, as described in the following definition:
\begin{defn}[Conway, \cite{Con01}] \label{defs-def-1}
(1) Let $L$ and $R$ be two sets of numbers. If there do not exist $a \in L$ and $b \in R$ such that $a \geq b$, there is a number denoted as $\{L\mid R\}$ with some name $x$.
(2) For every number named $x$, there is at least one pair of sets of numbers $(L_x,R_x)$ such that $x = \{L_x\mid R_x\}$.
\end{defn}
As is suggested by their names, $L_x$ is the left set of $x$, and $R_x$ is the right set of $x$. Conway also represents the number $x = \{L_x\mid R_x\}$ as $x = \{x^L\mid x^R\}$, where the left options $x^L$ and right options $x^R$ run through all members of $L_x$ and $R_x$, respectively. (We explain the meaning of assigning a name $x$ to a form $\{L \mid R\}$ later.) Having introduced the construction of numbers, we now consider properties of surreals. Let $\mathbf{No}$ be the class of surreal numbers, and for all $a \in \mathbf{No}$, let $\mathbf{No}_{<a}$ be the class of numbers $<a$ and $\mathbf{No}_{>a}$ be the class of numbers $>a$. The following are the basic arithmetic properties of numbers:
\begin{defn}[Conway, \cite{Con01}] \label{defs-def-2}
Let $x_1,x_2 \in \mathbf{No}$. Then,

\begin{enumerate}\setlength{\itemsep}{0mm}
    \item[\rm 1.] {\bf Comparison:} $x_1 \leq x_2$ iff (no $x_1^L \geq x_2$ and no $ x_2^R \leq x_1$); $x_1 \geq x_2$ iff $x_2 \leq x_1$; $x_1=x_2$ iff $(x_1 \geq x_2$ and $x_1 \leq x_2)$; $x_1 < x_2$ iff $(x_1 \leq x_2$ and $x_1 \not\geq x_2)$; $x_1 > x_2$ iff $x_2 < x_1$.
    \item[\rm 2.] {\bf Negation:} $-x_1 = \{-x_1^R\mid -x_1^L\}$.
    \item[\rm 3.] {\bf Addition:} $x_1+x_2=\{x_1^L +x_2, x_1+x_2^L\mid x_1^R+x_2, x_1+x_2^R\}$.
    \item[\rm 4.] {\bf Multiplication:} $x_1 \times x_2 = \{x_1^Lx_2+x_1x_2^L-x_1^Lx_2^L, x_1^Rx_2+x_1x_2^R-x_1^Rx_2^R\mid \\ x_1^Lx_2+x_1x_2^R-x_1^Lx_2^R, x_1^Rx_2+x_1x_2^L-x_1^Rx_2^L\}$.

\end{enumerate}
\end{defn}
The first part of Definition~\ref{defs-def-2} yields the following theorem relating $x, x^L,$ and $x^R$:
\begin{theorem}[Conway, \cite{Con01}] \label{defs-thm-1}
For all $x \in \mathbf{No},$ $x^L < x < x^R$.
\end{theorem}
Let $\mathbf{On}$ be the class of ordinals, and for all $\alpha \in \mathbf{On}$, let $\mathbf{On}_{<\alpha}$ be the set of ordinals $<\alpha$ and $\mathbf{On}_{>\alpha}$ be the class of ordinals $>\alpha$. Of relevance is the fact that $\mathbf{On} \subsetneq \mathbf{No}$; in particular, if $\alpha \in \mathbf{On}$, $\alpha$ has the representation $\alpha = \{\mathbf{On}_{<\alpha} \mid \}$. Ordinals can be combined to yield ``infinite numbers,'' and the multiplicative inverses of such numbers are ``infinitesimal numbers.'' Representations of the form $\{L\mid R\}$ are known as genetic formulae. The name ``genetic formula'' highlights the fact that numbers can be visualized as having birthdays; i.e.~for every $x \in \mathbf{No}$, there exists $\alpha \in \mathbf{On}$ such that $x$ has birthday $\alpha$ (Theorem 16 of Conway's book~\cite{Con01}). We write $\mathfrak{b}(x) = \alpha$ if the birthday of $x$ is $\alpha$. Because of the birthday system, the form $\{L\mid R\}$ represents a unique number:
\begin{theorem}[Conway, \cite{Con01}] \label{defs-thm-2}
Let $x\in\mathbf{No}$. Then, $x = \{L_x\mid R_x\}$ iff $x$ is the oldest number greater than the elements of $L_x$ and less than the elements of $R_x$.
\end{theorem}
Because numbers are defined recursively, mathematical induction can be performed on the birthdays of numbers; i.e.\ we can hypothesize that a statement holds for older numbers and use that hypothesis to show that the statement holds for younger numbers.

We now demonstrate how a form $\{L \mid R\}$ can have a name $x$. To do this, we need two tools: (1) Basic arithmetic properties; and (2) $\mathbb{R} \subsetneq \mathbf{No}$. The first tool is established in Definition~\ref{defs-def-2}. The second tool results from the construction of numbers. We now illustrate the construction process:

\begin{enumerate}\setlength{\itemsep}{0mm}
\item \emph{Day $0$}: $0 = \{\mid\}$ is taken as a ``base case'' for the construction of other numbers.
\item \emph{Day $1$}: $0$ can belong in either the left or right set of a new number, so we get two numbers, named so: $1 = \{0\mid\}$ and $-1 = \{\mid0\}$.
\item \emph{Day $2$}: We can now use the numbers $0, \pm 1$ in the left and right sets of newer numbers still, which we name so: $2=\{1\mid\}$, $1/2=\{0\mid 1\}$, $-2 = \{\mid -1\}$, and $-1/2 = \{-1 \mid 0\}$.
\item All dyadic rationals are created on finite days.
\item \emph{Day $\omega$}: All other reals are created, so $\mathbb{R} \subsetneq \mathbf{No}$.

\end{enumerate}

It is easy to show that for any $\{L \mid R\}, \{L' \mid R'\} \in \mathbf{No}$ such that their names, say $a = \{L \mid R\}, b = \{L' \mid R'\}$, satisfy $a,b \in \mathbb{R}$, an arithmetic property governing the reals $a,b$ also holds for the surreals $\{L \mid R\}, \{L' \mid R'\}$. For example, the sum of two surreals equals the sum of their names; i.e.~$\{0 \mid\} = 1$ and $\{1 \mid \} = 2$, so $\{0 \mid\} + \{0 \mid\} = 1 + 1 = 2 = \{1 \mid\}$, where we use part $3$ of Definition~\ref{defs-def-2} to do the surreal addition. Thus, assigning names like $1$ to $\{0 \mid\}$ and $2$ to $\{1 \mid\}$ makes sense.

Theorem 21 from~\cite{Con01} states that every number can be uniquely represented as a formal sum over ordinals $\sum_{i \in \mathbf{On}_{<\beta}} r_i \cdot \omega^{y_i}$, where the coefficients $r_i$ satisfy $r_i \in \mathbb{R}$ and the numbers $y_i$ form a decreasing sequence.\footnote{The method by which these transfinite sums are evaluated is not relevant to the rest of the paper but is discussed thoroughly in~\cite{Con01}.} This representation is called the \emph{normal form} of a number.

\subsection{Gaps} \label{gapsadd}
\noindent Suppose we have already constructed $\mathbf{No}$. Unlike its real analogue, the surreal number line is riddled with gaps, which are defined as follows:
\begin{defn}[Conway,~\cite{Con01}] \label{defs-def-4}
Let $L$ and $R$ be two classes of numbers such that $L \cup R = \mathbf{No}$. If there do not exist $a \in L$ and $b \in R$ such that $a \geq b$, the form $\{L \mid R\}$ represents a gap.
\end{defn}
Gaps are Dedekind sections of $\mathbf{No}$, and in the language of birthdays, all gaps are born on day $\mathbf{On}$. Notice that gaps are distinct from numbers because if $\{L \mid R\}$ is the representation of a gap, then there cannot be any numbers between the elements of $L$ and the elements of $R$. The Dedekind completion of $\mathbf{No}$, which contains all numbers and gaps, is denoted $\mathbf{No}^\mathfrak{D}$. Basic arithmetic operations (except for negation) on $\mathbf{No}^\mathfrak{D}$ are different from those on $\mathbf{No}$, as discussed in~\cite{For04}. Three gaps worth identifying are (1) $\mathbf{On} = \{\mathbf{No}\mid \}$, the gap larger than all surreals (surreal version of infinity); (2) $\mathbf{Off} = -\mathbf{On}$, the gap smaller than all surreals (surreal version of $-$infinity); and (3) $\infty = \{\mathrm{+ finite\,\,and - numbers}\mid \mathrm{+ infinite\,\,numbers}\}$, the object called infinity and denoted $\infty$ in real analysis. The gap $\mathbf{On}$ is important for the purpose of evaluating limits of sequences and functions. Throughout the rest of the paper, we say that \emph{a sequence is of length $\bm{On}$} if its elements are indexed over all elements of the proper class of ordinals $\mathbf{On}$.

Gaps can be represented using normal forms as well~\cite{Con01}. All gaps can be classified into two types, Type I and Type II, which have the following normal forms:
\begin{align*}
\mathrm{Type~I:}~&\sum_{i \in \mathbf{On}} r_i \cdot \omega^{y_i}; \\
\mathrm{Type~II:}~&\sum_{i \in \mathbf{On}_{<\alpha}} r_i \cdot \omega^{y_i} \oplus \left(\pm \omega^{\Theta}\right),
\end{align*}
where in both sums the $r_i$ are nonzero real numbers and $\{y_i\}$ is a decreasing sequence. In the Type II sum, $\alpha \in \mathbf{On}$, $\Theta$ is a gap whose right class contains all of the $y_i$, and the operation $\oplus$ denotes the sum of a number $n$ and gap $g$, defined by $n \oplus g = \{n+g^\mathscr{L} \mid n+g^\mathscr{R}\}$. Also, $\omega^{\Theta} = \{0, a \cdot \omega^l \mid b \cdot \omega^r\}$, where $a,b \in \mathbb{R}_{>0}$ and $l \in \mathscr{L}_{\Theta}, r \in \mathscr{R}_{\Theta}$.

The real number line does not have gaps because it is Dedekind complete. The topology that we present in Subsection~\ref{deffuncs} allows us to obtain results in surreal analysis, like the Intermediate Value Theorem and the Extreme Value Theorem, that are analogous to those in real analysis even though $\mathbf{No}$ is not Dedekind complete.

\subsection{Functions} \label{deffuncs}
\noindent Functions on $\mathbf{No}$ also exist; namely, for $A \subset \mathbf{No}$, a function $f : A \to \mathbf{No}$ is an assignment to each $x \in A$ a unique value $f(x) \in \mathbf{No}$. It is important for the purpose of studying surreal analysis to define what it means for a surreal function to be continuous. To this end, we define a topology on $\mathbf{No}$ as follows. We first define what it means to have a topology on $\mathbf{No}$:

\begin{defn}\label{top}
A topology on $\mathbf{No}$ is a collection $\mathcal{A}$ of subclasses of $\mathbf{No}$ satisfying the following properties:
\begin{enumerate}
\item $\varnothing, \mathbf{No} \in \mathcal{A}$.
\item $\bigcup_{i \in I} A_i \in \mathcal{A}$ for any subcollection $\{A_i\}_{i \in I} \subset \mathcal{A}$ indexed over a proper set $I$.
\item $\bigcap_{i \in I} A_i \in \mathcal{A}$ for any subcollection $\{A_i\}_{i \in I} \subset \mathcal{A}$ indexed over a finite set $I$.
\end{enumerate}
The elements of $\mathcal{A}$ are declared to be ``open.''
\end{defn}

\begin{remark}
In Definition~\ref{top}, it is important to note that not all unions of open subclasses of $\mathbf{No}$ are necessarily open; only those that are indexed over a \emph{proper set} need to be open. This stipulation is crucial for two purposes: (1) to make $\mathbf{No}$ connected; and (2) to make the compactness arguments in Subsection~\ref{integralFTC} work.

The set-theoretic details of our constructions are not of particular importance to the remainder of this article. (As Conway observes in~\cite[p.\ 66]{Con01}, we ought to be free to construct new objects from previously constructed ones without fear.) However, we point out that there are many ways of justifying our constructions from a set-theoretic point of view. As Ehrlich has shown in \cite{Ehr88}, it is possible to construct a model of the collection $\mathbf{No}$ of surreal numbers in the von Neumann-Bernays-G\"odel (NBG) set theory. This model suffices for our purposes, although a bit of care must be taken in Definition~\ref{top}. We are not free to put the open classes of a topology into a class; instead, we label the open classes, and we check that the labelled classes satisfy the axioms of a topology: i.e.\ the empty class and the class of all surreals are open, the union indexed over a set of open classes is open, and the intersection of two (or finitely many) open classes is open. Similarly, when we refer to a ``covering'' of a subclass of $\mathbf{No}$, we mean a collection of labels assigned to classes. This does not pose any problems, because we do not do anything of set-theoretic interest with the elements of our topology $\mathcal{A}$.
\end{remark}

We next define what we want an open subclass of $\mathbf{No}$ to be:

\begin{defn}\label{defopen}
The empty set is open. A nonempty subinterval of $\mathbf{No}$ is open if it (1) has endpoints in $\mathbf{No} \cup \{\mathbf{On}, \mathbf{Off}\}$; and (2) does not contain its endpoints.\footnote{We use the notation $(a,b)$ to denote an interval not containing its endpoints and the notation $[a,b]$ to denote an interval containing its endpoints.} A subclass $A \subset \mathbf{No}$ is open if it has the form $A = \bigcup_{i\in I} A_i$, where $I$ is a proper set and the $A_i$ are open intervals.
\end{defn}

\begin{remark}
The open classes are the surreal analogues of the open sets discussed in real analysis. Also, observe that our notion of openness is not equivalent to local openness; i.e.~it is not equivalent to the following statement: ``A space $S$ is open if every point in $S$ has a neighborhood contained in $S$.'' For example, the interval $(\infty, \mathbf{On})$ does satisfy the requirement that every point in the interval has a neighborhood in it, but according to our definition, it is not open. However, our notion of openness indeed does \emph{imply} local openness.
\end{remark}

\begin{example}\label{openexamp}
Now that we have specified the open subclasses of $\mathbf{No}$, we provide two examples of a union of open subclasses of $\mathbf{No}$: one that is open, and one that is not. Notice that the interval $(\mathbf{Off}, \infty)$ is open, even though it has $\infty$ as an endpoint, because it can be expressed as a union of open intervals over the (proper) set of integers: $(\mathbf{Off}, \infty)=\bigcup_{i\in \mathbb{Z}} (\mathbf{Off},i)$. On the other hand, we claim that the interval $(\infty, \mathbf{On})$ is not open because it cannot be expressed as a union of open intervals over a proper set. Indeed, suppose there exists a collection $\{A_i\}_{i \in I}$ of open subintervals of $\mathbf{No}$ with endpoints in $\mathbf{No} \cup \{\mathbf{On}\}$ indexed over a proper set $I$ such that $(\infty, \mathbf{On}) = \bigcup_{i \in I} A_i$. Then, consider $x = \{0,1,2,\dots \mid \{\inf A_i\}_{i \in I}\}$. By construction, $x \in \mathbf{No}$, $x > \infty$, but $x \not\in \bigcup_{i \in I} A_i$, contradicting $(\infty, \mathbf{On}) = \bigcup_{i \in I} A_i$.
\end{example}

\begin{proposition}\label{topproof}
Definition~\ref{defopen} defines a topology on $\mathbf{No}$.
\end{proposition}
\begin{proof}
Consider the collection $\mathcal{A}$ of open subclasses of $\mathbf{No}$. By definition, $\varnothing \in \mathcal{A}$ and since $\mathbf{No} = (\mathbf{Off}, \mathbf{On})$ has endpoints in $\mathbf{No} \cup \{\mathbf{On}, \mathbf{Off}\}$, it is also open. Because the union over a proper set of unions over proper sets is itself a union over a proper set, any union over a proper set of elements of $\mathcal{A}$ is itself an element of $\mathcal{A}$. By induction, it then suffices to show that if $\{A_\alpha\}_{\alpha \in A}$ and $\{B_\beta\}_{\beta \in B}$ are collections of open intervals indexed by proper sets $A,B$, then $C = \bigcup_{\alpha \in A} A_\alpha \cap \bigcup_{\beta \in B} B_\beta \in \mathcal{A}$. But we have that $$\bigcup_{\alpha \in A} A_\alpha \cap \bigcup_{\beta \in B} B_\beta = \bigcup_{(\alpha, \beta) \in A \times B} A_\alpha \cap B_\beta,$$
and the expression on the right-hand-side is a union over a proper set of open intervals, so indeed $C \in \mathcal{A}$.
\end{proof}

We take the subspace topology to be as usual: if $X \subset \mathbf{No}$, then $X' \subset X$ is ``open in $X$'' if there exists open $X'' \subset \mathbf{No}$ such that $X' = X \cap X''$. We can now define what it means for a function to be \emph{continuous}, by which we mean continuous with respect to the topology of Definition~\ref{defopen}:

\begin{defn} \label{defcont}
Let $A \subset \mathbf{No}$, and let $f : A \to \mathbf{No}$ be a function. Then $f$ is \emph{continuous} on $A$ if for any class $B$ open in $\mathbf{No}$, $f^{-1}(B)$ is open in $A$.
\end{defn}

Because the surreal numbers themselves were constructed inductively with genetic formulas $\{L \mid R\}$, it is only natural to wonder whether surreal functions can be constructed in a similar way. The rest of this subsection describes the class of \emph{genetic} functions; i.e. functions that have inductive definitions.

We begin by presenting one method of constructing genetic functions defined on all of $\mathbf{No}$. Let $S$ be a class of genetic functions defined on all of $\mathbf{No}$. We will define a function $f : \mathbf{No} \to \mathbf{No}$ inductively as follows. The construction process can proceed in one of two ways:

\begin{itemize}
\item Consider the Ring (capitalized because $\mathbf{No}$ is a proper class) $K$ obtained by adjoining to $\mathbf{No}$ the following collection of symbols: $\{g(a),g(b) : g \in S \cup \{f\}\}$, where $a,b$ are indeterminates (notice that the symbols $f(a), f(b)$ are allowed, but no other symbols involving $f$ are allowed). We then obtain a class of symbols $S(a,b) = \{c_1h(c_2x+c_3)+c_4 : c_1,c_2,c_3,c_4 \in K, h \in S\}$ (now notice that $h \in S$, so we cannot take $h$ to be $f$ in this part of the construction). Next, consider the Ring $R(a,b)$ generated over $\mathbf{No}$ by adjoining the elements of $S(a,b)$, and let $L_f,R_f \subset R(a,b)$ be proper subsets. Fix $x \in \mathbf{No}$, and suppose that $f(y)$ has already been defined for all $y \in L_x \cup R_x$. Also let $x^L \in L_x, x^R \in R_x$. Then replace $a$ with $x^L$ and $b$ with $x^R$ in $R(a,b)$ and consider the resulting sets of functions $L_f(x^L, x^R), R_f(x^L, x^R)$ from $\mathbf{No} \to \mathbf{No}$ (these sets are obtained from $L_f, R_f$ by substitution). Now, if for all $x^L,{x^L}' \in L_x$, $x^R,{x^R}' \in R_x$, $f^L \in L_f(x^L, x^R)$, $f^R \in R_f({x^L}',{x^R}')$ we have $f^L(x) < f^R(x)$, we let $f(x)$ be given by the expression
$$\left\{\bigcup_{x^L \in L_x, x^R \in R_x} \{f^L(x) : f^L \in L_f(x^L, x^R)\} \middle\rvert \bigcup_{x^L \in L_x, x^R \in R_x} \{f^R(x) : f^R \in R_f(x^L, x^R)\} \right\}.$$
In this case, $f$ is genetic. The elements of $L_f$ are called left options of $f$ and denoted $f^L$, and the elements of $R_f$ are called right options of $f$ and denoted $f^R$.
\item Let $g,h \in S$. Define $f = g \circ h$ by $f(x) = g(h(x))$. If for all $x \in \mathbf{No}$ we have $g(x) = \{L_g(x^L, x^R) \mid R_g(x^L, x^R)\}$ and $h(x) = \{L_h(x^L, x^R) \mid R_h(x^L, x^R)\}$, we have that $f(x)$ is given by the expression
\begin{align*}
& \left\{\bigcup_{x^L \in L_x, x^R \in R_x}\,\,\,\bigcup_{h^L \in L_h(x^L, x^R), h^R \in R_h(x^L, x^R)} \{g^L(h(x)) : g^L \in L_g(h^L(x), h^R(x))\} \middle\rvert \right.\\ &\left. \bigcup_{x^L \in L_x, x^R \in R_x}\,\,\,\bigcup_{h^L \in L_h(x^L, x^R), h^R \in R_h(x^L, x^R)} \{g^R(h(x)) : g^R \in L_g(h^L(x), h^R(x))\}\right\}.
\end{align*}
\end{itemize}

\begin{example}\label{funk1}
All polynomial functions are genetic and can be constructed using the method of the previous paragraph. For instance, let us illustrate the construction of $f(x) = x^2$. Let $S= \{\mathrm{id}\}$ where $\mathrm{id} : \mathbf{No} \to \mathbf{No}$ is the identity function. Then the Ring obtained by adjoining the collection of symbols $\{g(a), g(b) : g \in S \cup \{f\}\}$ to $\mathbf{No}$ contains the symbols $2a,2b,-a^2,-b^2,a+b$ and $-ab$. Therefore, the class $S(a,b)$ (and consequently the ring $R(a,b)$) contains the symbols $2xa - a^2, 2xb - b^2$, and $xa + xb - ab$. Let $L_f = \{2xa-a^2,2xb-b^2\}$ and $R_f = \{xa + xb - ab\}$. Fix $x \in A$, and suppose that $f$ has been defined on all $y \in L_x \cup R_x$. We then have that $L_f(x^L, x^R) = \{2xx^L - {x^L}^2, 2xx^R-{x^R}^2\}$ and $R_f(x^L, x^R) = \{xx^L+xx^R-x^Lx^R\}$. One can use the following inequalities to verify that for all $x^L,{x^L}' \in L_x$, $x^R,{x^R}' \in R_x$, $f^L \in L_f(x_L, x_R)$, $f^R \in R_f({x^L}',{x^R}')$ we have $f^L(x) < f^R(x)$:
$$(x-x^L)^2 > 0, (x - x^R)^2 > 0, \text{ and } (x - x^L)(x^R - x) > 0.$$
We then say that the function $f(x) = x^2$ is represented as $f(x) = \{2xx^L-{x^L}^2, 2xx^R-{x^R}^2\mid xx^L+xx^R-x^Lx^R\}$. In this case, the ``left options'' can be either $2xa-a^2$ or $2xb-b^2$ and the ``right options'' can be $xa+xb-ab$.

In general, if $f,g$ are functions constructed in the above way, then $f+g, fg, f \circ g$ are also constructible by the above process.
\end{example}

The reason that we only consider functions defined on all of $\mathbf{No}$ in the above construction process is that otherwise, we may end up evaluating functions at numbers outside of their domains. For example, if a function $f$ is only defined on the interval $[1/2,3/2]$, then it is not clear what $f(x^L)$ is when $x = 1 = \{0 \mid \}$. However, it is still possible to find genetic definitions of functions whose domains are proper subsets of $\mathbf{No}$, although we cannot necessarily write out a genetic formula for such functions. For example, a genetic definition of the function $f(x) = 1/x$ is given in~\cite{Con01}, but this definition cannot be stated purely in terms of symbols; it gives a genetic formula which depends on the stipulation that we must ``ignore options'' when division by 0 occurs. In Section~\ref{newfunc}, we provide examples of two functions whose genetic definitions require options to be ``ignored'' when they satisfy certain conditions.

Genetic functions need not be continuous, even if they are constructed using continuous functions. Since we have an inductive definition for $f(x) = 1/x$, all rational functions are genetic, so in particular, $\frac{x}{1+x^2}$ is genetic. Then the unit step function
$$f(x) = \begin{cases} 0 & x < 0 \\ 1 & x \geq 0 \end{cases}$$
is genetic, with the following simple genetic definition:
$$\left\{\frac{x}{1+x^2} \middle\rvert \right\}.$$

\section{Two New Surreal Functions}\label{newfunc}

\noindent In this section, we propose new genetic definitions of two functions, namely $\arctan(x)$ and $\mathrm{nlog}(x)=-\log(1-x)$. We show that our definitions match their real analogues on their domains, and we also provide an example of how to evaluate our functions at non-real values.

Maclaurin expansions are useful for defining surreal functions, as discussed in Gonshor's book~\cite{Gon86}. Given $n \in \mathbb{N}$, $x \in \mathbf{No}$, and a real analytic function $f(x)$, let $[x]_n$ denote the $n$-truncation of the Maclaurin expansion of $f(x)$; i.e.~$[x]_n = \sum_{i=0}^n \frac{f^{(i)}(0) x^i}{i!}$. To avoid confusion with the corresponding functions on $\BR$, we will denote by $\ul{\arctan}(x), \ul{\nlog}(x)$ our definitions of $\arctan,\nlog$ respectively.

\begin{defn} \label{funky}
On all of $\mathbf{No}$ we define the following function:
{\footnotesize
\begin{align*}
\ul{\arctan}(x)  = &\left\{\frac{-\pi}{2}, \ul{\arctan}(x^L) +\left[\frac{x-x^L}{1+xx^L}\right]_{4n-1}, \ul{\arctan}(x^R)+\left[\frac{x-x^R}{1+xx^R}\right]_{4n+1} \middle\rvert \right. \\ & \left. \ul{\arctan}(x^R)-\left[\frac{x^R-x}{1+xx^R}\right]_{4n-1}, \ul{\arctan}(x^L)-\left[\frac{x^L-x}{1+xx^L}\right]_{4n+1}, \frac{\pi}{2} \right\}
\end{align*}} under three conditions: (1) if $y \in L_x \cup R_x$ is such that $$\left|\frac{x-y}{1+xy}\right| > 1,$$ we ignore the options in the formula that involve $y$; (2) if $x^L$ is such that
$$\left|\ul{\arctan}(x^L) \pm \left[\frac{\pm(x-x^L)}{1+xx^L}\right]_{4n \mp 1}\right| > \frac{\pi}{2},$$ then we ignore the options in the formula that involve $x^L$; and (3) if $x^R$ is such that $$\left|\ul{\arctan}(x^R) \pm \left[\frac{\pm(x-x^R)}{1+xx^R}\right]_{4n \pm 1}\right| > \frac{\pi}{2},$$ then we ignore the options in the formula that involve $x^R$. On the domain $(-\infty, 0]$ we define the following function
{\footnotesize
\begin{align*}
\ul{\mathrm{nlog}}(x)  = & \left\{\ul{\nlog}(x^L)+\left[\frac{x-x^L}{1-x^L}\right]_{n}, \ul{\nlog}(x^R)+\left[\frac{x-x^R}{1-x^R}\right]_{2n+1}  \middle\rvert \right. \\ & \left. \ul{\nlog}(x^R)-\left[\frac{x^R-x}{1-x}\right]_{n}, \ul{\nlog}(x^L)-\left[\frac{x^L-x}{1-x}\right]_{2n+1} \right\}
\end{align*}} where if $y \in L_x \cup R_x$ is such that $$\left|\frac{x-y}{1-y}\right| \geq 1 \text{ or } \left|\frac{x-y}{1-x}\right| \geq 1,$$ we ignore the options in the formula that involve $y$.
\end{defn}

\begin{remark}
As discussed in Subsection~\ref{deffuncs}, the above two definitions require ``verbal'' conditions in addition to formulas in order to be stated completely. For example, let us consider $\ul{\arctan}$. The verbal conditions for this function tell us to ignore options in the genetic formula when they satisfy at least one of three inequalities. It is then natural to wonder whether for some $x$ these conditions eliminate all possible options in $L_x$ or in $R_x$ and thereby cause the definition to give a ``wrong answer'': indeed, if $x = 1/2$ and all right options are eliminated, then $\ul{\arctan}(1/2)$ would be equal to $\ul{\arctan}(1),$ which is an undesirable result. As it happens, this problem does not occur in either the case of $\ul{\arctan}$ or $\ul{\nlog}$; i.e. if we write $x = \{L_x \mid R_x\}$ and the sets $\wt{L}_x, \wt{R}_x$ are the result of eliminating all options that satisfy the verbal conditions in the definition of either $\ul{\arctan}$ or $\ul{\nlog}$, then $x = \{\wt{L}_x \mid \wt{R}_x\}$.

Also, the above genetic formula of $\ul{\nlog}$ fails to work outside of the domain $(-\infty, 0]$. Indeed, if $x = \{L_x \mid R_x\} \in \mathbf{No}$ lies outside of this domain, then the sets $\wt{L}_x, \wt{R}_x$ that result from eliminating all options that satisfy the verbal condition in the definition of $\ul{\nlog}$ may not be such that $x = \{\wt{L}_x \mid \wt{R}_x\}$.
\end{remark}

To check that the two definitions are reasonable, we must verify that the real functions $\arctan$ and $\mathrm{nlog}$ agree with $\ul{\arctan}$ and $\ul{\nlog}$ when the argument is real. To this end, we state and prove the following theorem:

\begin{theorem} \label{funkies}
For all $x \in \BR$, we have $\arctan(x) = \ul{\arctan}(x)$, and for all $x \in \BR_{\leq 0}$, we have $\mathrm{nlog}(x) = \ul{\nlog}(x)$.
\end{theorem}
\begin{proof}
Notice that $\ul{\arctan}(0) = \arctan(0) = 0$ and $\ul{\mathrm{nlog}}(0) = \mathrm{nlog}(0) = 0$. We may now proceed by induction upon the options of $x$, since we have base cases for both functions. Thus, assume that $\ul{f}(x^L)= f(x^L)$ and $\ul{f}(x^R) = f(x^R)$, where $f$ is either $\arctan$ or $\nlog$.

Let us begin by considering $\arctan$. We will first show that $(\ul{\arctan}(x))^L < \arctan(x) < (\ul{\arctan}(x))^R$. It suffices to have the following four inequalities:
\begin{align*}
\left[\frac{x-x^L}{1+xx^L}\right]_{4n-1} & < \arctan(x) - \arctan(x^L)=\arctan\left(\frac{x-x^L}{1+xx^L}\right)   , \\ \left[\frac{x-x^R}{1+xx^R}\right]_{4n+1} & < \arctan(x)-\arctan(x^R)=\arctan\left(\frac{x-x^R}{1+xx^R}\right)   , \\ \left[\frac{x^R-x}{1+xx^R}\right]_{4n-1} & < \arctan(x^R) - \arctan(x)=\arctan\left(\frac{x^R-x}{1+xx^R}\right)  , \\ \left[\frac{x^L-x}{1+xx^L}\right]_{4n+1} & < \arctan(x^L) - \arctan(x)=\arctan\left(\frac{x^L-x}{1+xx^L}\right)  .
 \end{align*}
 From the Maclaurin series expansion of $\arctan(x)$, we know that $[z]_{4n-1} < \arctan(z)$ when $0 < z \leq 1$ and $[z]_{4n+1} < \arctan(z)$ when $-1 \leq z < 0$. So, it suffices to check the following inequalities:
  $$\left|\frac{x-x^L}{1+xx^L}\right| \leq 1 \text{ and } \left|\frac{x-x^R}{1+xx^R}\right| \leq 1.$$
  The above inequalities are precisely given by the extra conditions imposed on the options of $x$ in Definition~\ref{funky}. We next show that  $(\ul{\arctan}(x))^L$ and $(\ul{\arctan}(x))^R$ both ``approach'' $\arctan(x)$.\footnote{By ``approach,'' we mean ``approach as real sequences'' so that $\{(\ul{\arctan}(x))^L \mid (\ul{\arctan}(x))^R\}=\arctan(x)$.} If $L_x \neq \varnothing$, pick $x^L \in \BR$ such that $1 + xx^L > 0$, and if $R_x \neq \varnothing$, pick $x^R \in \BR$ such that $1 + xx^R > 0$ (observe that we can always make such choices). Since $\lim_{n \rightarrow \infty}  [z]_{4n-1} = \arctan(z)$ and $\lim_{n \rightarrow \infty} -[-z]_{4n+1} = \arctan(z)$ when $z \in \BR$ such that $0 < z \leq 1$, we have the following limits:
  \begin{eqnarray*}
  \lim_{n \rightarrow \infty} \ul{\arctan}(x^L) +\left[\frac{x-x^L}{1+xx^L}\right]_{4n-1} & = & \arctan(x^L) + \arctan\left(\frac{x-x^L}{1+xx^L}\right) \\
   & = & \arctan(x); \\
   \lim_{n \rightarrow \infty} \ul{\arctan}(x^R)+\left[\frac{x-x^R}{1+xx^R}\right]_{4n+1} & = & \arctan(x^R) + \arctan\left(\frac{x-x^R}{1+xx^R}\right) \\
   & = & \arctan(x).
  \end{eqnarray*}
  Similarly, since $\lim_{n \rightarrow \infty}  [z]_{4n+1} = \arctan(z)$ and $\lim_{n \rightarrow \infty} -[-z]_{4n-1} = \arctan(z)$ when $z \in \BR$ such that $-1 \leq z < 0$, we have the following limits:
   \begin{eqnarray*}
  \lim_{n \rightarrow \infty} \ul{\arctan}(x^L) -\left[\frac{x^L-x}{1+xx^L}\right]_{4n+1} & = & \arctan(x^L) - \arctan\left(\frac{x^L-x}{1+xx^L}\right)  \\
   & = & \arctan(x); \\
   \lim_{n \rightarrow \infty} \ul{\arctan}(x^R)-\left[\frac{x^R-x}{1+xx^R}\right]_{4n-1} & = & \arctan(x^R) - \arctan\left(\frac{x^R-x}{1+xx^R}\right) \\
   & = & \arctan(x).
  \end{eqnarray*}
   It follows that $\ul{\arctan}(x) = \{\{\arctan(x)-\frac{1}{n} : n \in \BZ_{>0}\} \mid \{\arctan(x) + \frac{1}{n} : n \in \BZ_{> 0}\}\} = \arctan(x)$.

Let us now consider $\nlog$. We will first show that $(\ul{\nlog}(x))^L < \nlog(x) < (\ul{\nlog}(x))^R$. It suffices to have the following four inequalities:
\begin{align*}
\left[\frac{x-x^L}{1-x^L}\right]_{n} & <  \mathrm{nlog}(x) - \mathrm{nlog}(x^L) = \mathrm{nlog}\left(\frac{x-x^L}{1-x^L}\right)\\ \left[\frac{x-x^R}{1-x^R}\right]_{2n+1} & <  \mathrm{nlog}(x)-\mathrm{nlog}(x^R) = \mathrm{nlog}\left(\frac{x-x^R}{1-x^R}\right)  \\ \left[\frac{x^R-x}{1-x}\right]_{n} & <   \mathrm{nlog}(x^R)-\mathrm{nlog}(x) = \mathrm{nlog}\left(\frac{x^R-x}{1-x}\right) \\ \left[\frac{x^L-x}{1-x}\right]_{2n+1} & <  \mathrm{nlog}(x^L)-\mathrm{nlog}(x) = \mathrm{nlog}\left(\frac{x^L-x}{1-x}\right).
 \end{align*}
 From the Maclaurin series expansion of $\nlog(x)$, we know that $[z]_{n} < \mathrm{nlog}(z)$ when $0 < z < 1$ and $[z]_{2n+1} < \mathrm{nlog}(z)$ when $-1 < z < 0$. So, it suffices to check the following inequalities:
  $$\left|\frac{x-x^L}{1-x^L}\right| < 1, \left|\frac{x-x^L}{1-x}\right| < 1, \left|\frac{x-x^R}{1-x^R}\right| < 1, \text{ and } \left|\frac{x-x^R}{1-x}\right| < 1.$$
The above four inequalities are precisely given by the extra conditions imposed on the options of $x$ in Definition~\ref{funky}. We next show that  $(\ul{\nlog}(x))^L$ and $(\ul{\nlog}(x))^R$ both ``approach'' $
\nlog(x)$. Since $\lim_{n \rightarrow \infty}  [z]_{n} = \nlog(z)$ and $\lim_{n \rightarrow \infty} [-z]_{2n+1} = \nlog(-z)$ when $z \in \BR$ such that $0 < z < 1$, we have the following limits:
\begin{eqnarray*}
  \lim_{n \rightarrow \infty} \ul{\nlog}(x^L) +\left[\frac{x-x^L}{1-x^L}\right]_{n} & = & \nlog(x^L) + \nlog\left(\frac{x-x^L}{1-x^L}\right)  \\
   & = & \nlog(x) ; \\
   \lim_{n \rightarrow \infty} \ul{\nlog}(x^R)+\left[\frac{x-x^R}{1-x^R}\right]_{2n+1} & = & \nlog(x^R) + \nlog\left(\frac{x-x^R}{1-x^R}\right)  \\
   & = & \nlog(x) ; \\
  \lim_{n \rightarrow \infty} \ul{\nlog}(x^L) -\left[\frac{x^L-x}{1-x}\right]_{2n+1} & = & \nlog(x^L) - \nlog\left(\frac{x^L-x}{1-x}\right)  \\
   & = & \nlog(x); \\
   \lim_{n \rightarrow \infty} \ul{\nlog}(x^R)-\left[\frac{x^R-x}{1-x}\right]_{n} & = & \nlog(x^R) - \nlog\left(\frac{x^R-x}{1-x}\right) \\
   & = & \nlog(x).
  \end{eqnarray*}
   It follows that $\ul{\nlog}(x) = \{\{\nlog(x)-\frac{1}{n} : n \in \BZ_{>0}\} \mid \{\nlog(x) + \frac{1}{n} : n \in \BZ_{>0}\}\} = \nlog(x)$.
\end{proof}

We now provide an example of how to use Definition~\ref{funky} in a computation:

\begin{example}\label{compexamp}
Consider $\omega = \{\BZ_{>0} \mid\}$, and let us evaluate $\ul{\arctan}(\omega)$. The genetic formula for $\ul{\arctan}$ gives the following:
$$\ul{\arctan}(\omega) = \left\{\frac{-\pi}{2}, \arctan(k) +\left[\frac{\omega-k}{1+k\omega}\right]_{4n-1}  \middle\rvert  \arctan(k)-\left[\frac{k-\omega}{1+k\omega}\right]_{4n+1}, \frac{\pi}{2} \right\},$$
where $k$ runs through the elements of $\BZ_{> 0}$ and where we have used Theorem~\ref{funkies} to say that $\ul{\arctan}(k) = \arctan(k)$ for $k \in \BZ_{> 0}$. Observe that we have the following equality:
$$\frac{\omega-k}{1+k\omega} = \frac{1}{k} - \frac{k^2 + 1}{k(k\omega + 1)}.$$
Since the $(4n-1)$-truncations of the Maclaurin series of $\arctan$ are increasing on the interval $(0,1)$, we have that
$$\left[\frac{1}{k}\right]_{4n-1} \geq \left[\frac{1}{k} - \frac{k^2 + 1}{k(k\omega + 1)}\right]_{4n-1}.$$
But we also have that $[x]_{4n-1} < \arctan(x)$ when $x \in (0,1)$, so we obtain the following inequality:
$$\left[\frac{1}{k}\right]_{4n-1} < \arctan\left(\frac{1}{k}\right) - \frac{1}{\infty}.$$
Combining our results, we have the following inequality:
$$\arctan(k) + \left[\frac{\omega - k}{1 + k\omega}\right]_{4n-1} < \arctan(k) + \arctan\left(\frac{1}{k}\right) - \frac{1}{\infty} = \frac{\pi}{2} - \frac{1}{\infty}.$$
But since $[x]_{4n-1} > 0$ for $x \in (0,1)$, we can make $$\frac{\pi}{2} - \left(\arctan(k) + \left[\frac{\omega - k}{1 + k\omega}\right]_{4n-1}\right) < \varepsilon$$
for each $n$ and for every real $\varepsilon > 0$ by taking $k$ sufficiently large. It follows that the left collection of $\ul{\arctan}(\omega)$ may be written as $\left\{\frac{\pi}{2} - \frac{1}{n} : n \in \BZ_{> 0}\right\}$.
A similar analysis of the right collection of $\ul{\arctan}(\omega)$ yields that all of the options of the form
$$\arctan(k) - \left[\frac{k - \omega}{1 + k\omega}\right]_{4n+1}$$
are greater than $\frac{\pi}{2}$, so the right collection of $\ul{\arctan}(\omega)$ may be written as $\left\{\frac{\pi}{2}\right\}$. It follows that $$\ul{\arctan}(\omega) = \left\{\frac{\pi}{2} - \frac{1}{n} \middle\rvert \frac{\pi}{2} \right\} = \frac{\pi}{2} - \frac{1}{\omega}.$$
\end{example}

As is detailed in~\cite{Van01} and~\cite{Van94}, it is possible to define a surreal extension of analytic functions, including the arctangent function on the restricted domain $(-\infty, \infty)$ and the $\nlog$ function on the domain $(-\infty, 1- \frac{1}{\infty})$ (the definitions given in~\cite{Van01} and~\cite{Van94} are not necessarily genetic, but are nonetheless interesting). It would be interesting to determine whether their functions agree with ours on the intersections of the domains of definition; however, it does not seem to be straightforward to check this in general. We begin by considering the $\arctan$ function. Let $x \in (-\infty, \infty)$, and observe that $x$ can be uniquely expressed as $r + \varepsilon$, where $r \in \BR$ and $\varepsilon$ is infinitesimal. Then if $p_r$ denotes the Taylor series expansion of the (real) arctangent function at $r$, then define a function $\ol{\arctan}(x) \defeq p_r(r+\varepsilon)$, where the method of computing $p_r$ is given in Chapter 4 of~\cite{Con01}. We may similarly define a function $\ol{\nlog}$ on the domain $(-\infty, 1- \frac{1}{\infty})$. We then want to determine whether $\ul{\arctan} = \ol{\arctan}$ on $(\infty, \infty)$ and whether $\ul{\nlog} = \ol{\nlog}$ on $(\infty, 0]$. Theorem~\ref{funkies} guarantees that $\ul{\arctan} = \ol{\arctan}$ and $\ul{\nlog} = \ol{\nlog}$ on $\mathbb{R}$, for it can be easily seen $\ol{\arctan}(x) = \arctan(x)$ and $\ol{\nlog}(x) = \nlog(x)$ when $x \in \BR$. The following theorem establishes the equalities $\ul{\arctan}(x) = \ol{\arctan}(x)$ and $\ul{\nlog}(x) = \ol{\nlog}(x)$ for a certain proper class of values $x \in \mathbf{No}$:

\begin{theorem}\label{funkyadd}
Let $S \subset \mathbf{No}$ be the proper class of numbers $x$ such that either $L_x = \{0\}$ or $R_x = \{0\}$. Then, $\ul{\arctan}(x) = \ol{\arctan}(x)$ for all $x \in S$ and $\ul{\nlog}(x) = \ol{\nlog}(x)$ for all $x \in S \cap \mathbf{No}_{\leq 0}$.
\end{theorem}
\begin{proof}
We first consider the $\arctan$ function. Suppose $L_x = \{0\}$. Then we have that $\arctan(x)$ is given by the expression
\begin{align*}
\ul{\arctan}(x)  = & \left\{\frac{-\pi}{2}, \left[x\right]_{4n-1}, \ul{\arctan}(x^R)+\left[\frac{x-x^R}{1+xx^R}\right]_{4n+1} \middle\rvert  \right. \\ & \left.  \ul{\arctan}(x^R)-\left[\frac{x^R-x}{1+xx^R}\right]_{4n-1}, \left[x\right]_{4n+1}, \frac{\pi}{2} \right\}.
\end{align*}
Notice that each $x \in S$ has normal form $x = r_x \cdot \omega^{-y_x}$, where $r_x \in \mathbb{R}$ and $y_x \in \mathbf{On}$. Pick $a,b \in R_x$. Because the Maclaurin truncations in the definition of the $\arctan$ function involve only finite powers, we obtain the following bound:
$$\mathfrak{b}\left(\left[\frac{x-a}{1+xa}\right]_{4n+1} + \left[\frac{b-x}{1+xb}\right]_{4n-1}\right) < z_x,$$ where $z_x$ denotes the smallest limit ordinal that is larger than $\mathfrak{b}\left(\omega^{-y_x}\right)$. Notice that for all $\alpha \in \mathbf{On}_{< z_x}$, we can make $$\big|[x]_{4n-1} - [x]_{4n+1}\big| < \omega^{-\alpha}$$ by taking $n$ sufficiently large. It follows that for each choice of $a,b \in R_x$, we can take $n$ sufficiently large so that we have the inequality
$$\big|[x]_{4n-1} - [x]_{4n+1}\big| < \left|\left[\frac{x-a}{1+xa}\right]_{4n+1} + \left[\frac{b-x}{1+xb}\right]_{4n-1}\right|.$$
Thus, in our expression for $\ul{\arctan}(x)$, we may simply throw out all options involving $x^R$. Thus, we find that 
$$\ul{\arctan}(x) = \left\{\frac{-\pi}{2}, \left[x\right]_{4n-1}  \middle\rvert    \left[x\right]_{4n+1}, \frac{\pi}{2} \right\},$$
and the expression on the right-hand-side of the above equality is precisely equal to $\ol{\arctan}(x)$ when $x \in S$. A similar argument works to handle the case when $R_x = \{0\}$.

We next consider the $\nlog$ function. Suppose $R_x = \{0\}$ (we need not consider the case of $L_x = \{0\}$ because the domain of the $\nlog$ function does not contain positive numbers). Then we have that $\nlog(x)$ is given by the expression
\begin{align*}
\ul{\mathrm{nlog}}(x)  =  \left\{\left[x\right]_{n}, \ul{\nlog}(x^L)+\left[\frac{x-x^L}{1-x^L}\right]_{2n+1}  \middle\rvert  \ul{\nlog}(x^L)-\left[\frac{x^L-x}{1-x}\right]_{n}, -\left[\frac{x}{x-1}\right]_{2n+1} \right\}
\end{align*}
Notice that each $x \in S$ has normal form $x = r_x \cdot \omega^{-y_x}$, where $r_x \in \mathbb{R}$ and $y_x \in \mathbf{On}$. Pick $a,b \in L_x$. Because the Maclaurin truncations in the definition of the $\nlog$ function involve only finite powers, we obtain the following bound:
$$\mathfrak{b}\left(\left[\frac{x-a}{1-a}\right]_{2n+1} + \left[\frac{b-x}{1-x}\right]_{n}\right) < z_x,$$ where $z_x$ denotes the smallest limit ordinal that is larger than $\mathfrak{b}\left(\omega^{-y_x}\right)$. Notice that for all $\alpha \in \mathbf{On}_{< z_x}$, we can make $$\left|[x]_{n} + \left[\frac{x}{x-1}\right]_{2n+1}\right| < \omega^{-\alpha}$$ by taking $n$ sufficiently large. It follows that for each choice of $a,b \in R_x$, we can take $n$ sufficiently large so that we have the inequality
$$\left|[x]_{n} + \left[\frac{x}{x-1}\right]_{2n+1}\right| < \left|\left[\frac{x-a}{1-a}\right]_{2n+1} + \left[\frac{b-x}{1-x}\right]_{n}\right|.$$
Thus, in our expression for $\ul{\nlog}(x)$, we may simply throw out all options involving $x^L$. Thus, we find that 
$$\ul{\nlog}(x) = \left\{\left[x\right]_{n} \middle\rvert -\left[\frac{x}{x-1}\right]_{2n+1} \right\},$$
and the expression on the right-hand-side of the above equality is precisely equal to $\ol{\nlog}(x)$ when $x \in S$.
\end{proof}

\begin{remark}
It can be deduced from the proof of Theorem~\ref{funkyadd} that $\ul{\arctan}\left(\frac{1}{\omega}\right)$ is given by the following normal form:
$$\ul{\arctan}\left(\frac{1}{\omega}\right) = \sum_{i = 1}^\infty \frac{(-1)^i \cdot \omega^{-i}}{2i-1}.$$
Then from the result of Example~\ref{compexamp}, we have that $\ul{\arctan}(\omega) + \ul{\arctan}\left(\frac{1}{\omega}\right) \neq \frac{\pi}{2}$, whereas for all $x \in \BR_{> 0}$ we have that $\arctan(x) + \arctan\left(\frac{1}{x}\right) = \frac{\pi}{2}$, thus providing evidence that the functional equation for the arctangent function does not necessarily extend from the reals to the surreals.

Moreover, our method of finding a genetic formula for a surreal extension of a (real) function cannot necessarily be applied to all real analytic functions. An important feature of the (real) functions $\arctan$ and $\nlog$ that allows us to construct surreal extensions with genetic formulas is that they satisfy simple functional equations. Specifically,
$$\arctan(a) + \arctan(b) = \arctan\left(\frac{a+b}{1-ab}\right) \text{ and } \nlog(a) = \nlog(b) + \nlog\left(\frac{a-b}{1-b}\right).$$
For functions (of one variable or of many variables) that satisfy more complicated functional equations or no functional equations at all, it is more difficult to find surreal extensions with genetic formulas.
\end{remark}

\section{Sequences of Numbers and their Limits}\label{seqs}

\noindent In this section, we dicuss limits of sequences. We first explain why it is best to consider $\mathbf{On}$-length sequences. We then provide a tool (Dedekind representation) that we use to define the limit of an $\mathbf{On}$-length sequence and to give a complete characterization of convergent sequences.

\subsection{Finding a Suitable Notion of Limit}\label{priorat}
\noindent In earlier work on surreal calculus, sequences (and series, which are sequences of partial sums) are restricted to have limit-ordinal length (as opposed to having length $\mathbf{On}$). The ``need'' for such a restriction can be explained informally as follows. Suppose we have a sequence $\mathfrak{A} = a_1, a_2, \dots$ of length $\mathbf{On}$. It is possible that for every $m \in \mathbf{On}$, $ \mathfrak{b}(a_i) > m$ for all $i \in \mathbf{On}_{>n}$ and some $n\in \mathbf{On}$. In an attempt to create a genetic formula for the limit of $\mathfrak{A}$, we can write $\lim_{i \rightarrow \mathbf{On}} a_i = \{L\mid R\}$. But, because $\mathfrak{b}(a_i)$ can be made arbitrarily large by taking $i$ large enough, the elements of at least one of $L,R$ would depend on the options of all terms in some subsequence (with length $\mathbf{On}$) of $\mathfrak{A}$. So, the cardinality of at least one of $L,R$ would have initial ordinal that is not less than $\mathbf{On}$, implying that at least one of $L,R$ would be a proper class rather than a set. Thus, the genetic formula $\{L \mid R\}$ of $\lim_{i \rightarrow \mathbf{On}} a_i$ would fail to satisfy Definition~\ref{defs-def-1}. However, if $c_1, c_2, \dots$ is a sequence of length $\alpha$ where $\alpha$ is a limit-ordinal, then there exists $m \in \mathbf{On}$ such that for all $i \in \mathrm{On}_{< \alpha}$, $m > \mathfrak{b}(c_i)$, so $\mathfrak{b}(c_i)$ is bounded. Thus, in any reasonable genetic formula $\{L \mid R\}$ for $\lim_{i \rightarrow \alpha}c_i$, $L$ and $R$ would be small enough to be sets, and $\{L \mid R\}$ would satisfy Definition~\ref{defs-def-1}. It is for this reason that earlier work has found the need to restrict the length of sequences.

While it does preserve Conway's construction of numbers (Definition~\ref{defs-def-1}), restricting sequences to have limit-ordinal length prevents us from obtaining the standard $\varepsilon$-$\delta$ notion of convergence for surreal sequences, as illustrated by the following theorem:
\begin{theorem} \label{thm-priorat-1}
Let $b \in \mathbf{No}$. Then, there does not exist an eventually nonconstant sequence $\mathfrak{A} = t_1, t_2, \dots$ of length $\alpha$, where $\alpha$ is a limit-ordinal, such that for every (surreal) $\varepsilon > 0$, there is an $N \in \mathbf{On}_{< \alpha}$ satisfying $\left| t_n-b\right| <\varepsilon$ whenever $n\in \mathbf{On}_{>N} \cap \mathbf{On}_{<\alpha}$.
\end{theorem}
\begin{proof}
Suppose such a sequence $\mathfrak{A}$ exists, and assume without loss of generality that none of the $t_i$ are equal to $b$. (If any $t_i$ are equal to $b$, discard them; the remaining subsequence has the same limit as the original sequence.) Now let $z$ be the smallest ordinal such that $z > \sup\{\left|b\right|, \mathfrak{b}(t_1), \mathfrak{b}(t_2), \dots\}$, and let $\varepsilon = 1/\omega^z$. Then there exists some $N \in \mathbf{On}_{< \alpha}$ such that for all $n \in \mathbf{On}_{>N} \cap \mathbf{On}_{<\alpha}$, (1) $t_n \neq b$; and (2) $-\varepsilon = -1/\omega^z< t_n-b<1/\omega^z = \varepsilon$. ($t_n \neq b$ holds for all $n \in \mathbf{On}_{< \alpha}$.) When (1) and (2) are combined, either $(b-1/\omega^{z} < t_n < b)$ or $(b < t_n<b+1/\omega^{z})$ holds. Thus, in any genetic formula for $t_n$, there is at least one left or right option whose birthday is $\geq z$. So for all $n \in \mathbf{On}_{>N} \cap \mathbf{On}_{<\alpha}$, $\mathfrak{b}(t_n)\geq z$, which is a contradiction because we chose $z$ so that $\mathfrak{b}(t_n) < z$.
\end{proof}
From Theorem~\ref{thm-priorat-1}, it is clear that restricting sequences to be of limit ordinal length is not optimal because such sequences do not have surreal limits (by ``limit'' we mean the $\varepsilon$-$\delta$ notion). We must consider $\mathbf{On}$-length sequences in $\mathbf{No}$, not only for the above reasons, but also in light of work by Sikorski, who showed that in a field of character $\omega_{\mu}$ (which is an initial regular ordinal number), we need to consider sequences of length $\omega_{\mu}$ to obtain convergence for nontrivial sequences~\cite{Sik48}. Thus, in $\mathbf{No}$, which has character $\mathbf{On}$, we need to consider sequences of length $\mathbf{On}$.

As mentioned earlier in this subsection, any formula of the form $\{L \mid R\}$ for the limit of an $\mathbf{On}$-length sequence must allow at least one of $L,R$ to be a proper class. Since the representation of numbers by genetic formulas forces $L,R$ to be sets, we need a new representation of numbers. The following is a particularly useful one:

\begin{defn} \label{def-nums-1}
For $x \in \mathbf{No}^{\mathfrak{D}}$, the Dedekind representation of $x$ is $x = \{\mathbf{No}_{<x}\mid \mathbf{No}_{>x}\}.$
\end{defn}

When considering Dedekind representations, we use  the following notational conventions: $x=\{x^\mathscr{L} \mid x^\mathscr{R}\} = \{\mathscr{L}_x \mid \mathscr{R}_x\}$, where we write ``$\mathscr{L}$'',``$\mathscr{R}$'' instead of ``$L$'',``$R$'' to distinguish Dedekind representations from genetic formulas. The following proposition allows us to use Dedekind representations of numbers in all basic arithmetic operations:

\begin{proposition} \label{nums-thm-1}
Every property in Definition~\ref{defs-def-2} holds when the numbers $x_1 = \{L_{x_1} \mid R_{x_1}\}$ and $x_2 = \{L_{x_2}\mid R_{x_2}\}$ are written in their respective Dedekind representations.
\end{proposition}

\begin{proof}
The proof is a routine calculation, so we omit it.
\end{proof}

\subsection{Evaluation of Limits of Sequences}\label{eval}

The approach we take to defining the limit of an $\mathbf{On}$-length sequence is analogous to the method Conway uses in introducing the arithmetic properties of numbers in Chapter 0 of~\cite{Con01}. Specifically, we first define the limit of an $\mathbf{On}$-length sequence to be a certain Dedekind representation and then prove that this definition is a reasonable one; i.e.~show that it is equivalent to the usual $\varepsilon$-$\delta$ definition for sequences that approach numbers. Of course, we could have defined the limit of an $\mathbf{On}$-length sequence in the usual way with $\varepsilon$ and $\delta$, but our definition is more general because it works for sequences that approach gaps as well as numbers. It is also in the spirit of the subject for the limit of a sequence to be of the form ``$\{$left collection | right collection$\}$.''

\begin{defn} \label{def15}
Let $\mathfrak{A} = a_1, a_2, \dots$ be an $\mathbf{On}$-length sequence. Then, define:
\begin{equation}
\ell(\mathfrak{A}) \defeq \left \{  a : a < \sup \left ( \bigcup_{i \geq 1} \bigcap_{j \geq i} \mathscr{L}_{a_j} \right )  \middle\rvert  b : b > \inf \left ( \bigcup_{i \geq 1} \bigcap_{j \geq i} \mathscr{R}_{a_j} \right ) \right \} \label{seqs-eqn-1}
\end{equation}
\end{defn}

\begin{defn} \label{defs-def-5}
Let $\mathfrak{A} = a_1, a_2, \dots$ be an $\mathbf{On}$-length sequence. We say that the limit of $\mathfrak{A}$ is $\ell$ and write $\lim_{i \rightarrow \mathbf{On}} a_i = \ell$ if the expression on the right-hand-side of (\ref{seqs-eqn-1}) in Definition~\ref{def15} is a Dedekind representation and $\ell = \ell(\mathfrak{A})$.
\end{defn}

\begin{remark}
In the above definition, we make no distinction as to whether $\ell(\mathfrak{A})$ is a number or a gap. Definition~\ref{defs-def-5} holds in both cases, although we do not say that surreal sequences approaching gaps are \emph{convergent}, just like we do not say real sequences approaching $\pm \infty$ are convergent.
\end{remark}

Before we state and prove Theorem~\ref{seqs-thm-1}, which proves the equivalence of Definition~\ref{defs-def-5} with the standard $\varepsilon$-$\delta$ definition, we need the following lemma:

\begin{lemma} \label{seqs-lem-1}
Let $\mathfrak{A} = a_1, a_2, \dots$ be an $\mathbf{On}$-length sequence, and let $\mathfrak{B} = a_k, a_{k+1}, \dots$ be an $\mathbf{On}$-length sequence. Then $\ell(\mathfrak{B}) = \ell(\mathfrak{A})$.
\end{lemma}

\begin{proof}
We need to show that the following statements hold: (1) $\sup \left ( \bigcup_{i \geq k} \bigcap_{j \geq i} \mathscr{L}_{a_j} \right ) = \sup \left ( \bigcup_{i \geq 1} \bigcap_{j \geq i} \mathscr{L}_{a_j} \right )$ and (2) $\inf \left ( \bigcup_{i \geq k} \bigcap_{j \geq i} \mathscr{R}_{a_j} \right ) = \inf \left ( \bigcup_{i \geq 1} \bigcap_{j \geq i} \mathscr{R}_{a_j} \right ).$ We first prove (1). Let $M = \left ( \bigcup_{i \geq k} \bigcap_{j \geq i} \mathscr{L}_{a_j} \right )$ and $N = \left ( \bigcup_{i \geq 1} \bigcap_{j \geq i} \mathscr{L}_{a_j} \right )$. Note that $\bigcap_{j \geq i} \mathscr{L}_{a_j} \subseteq \bigcap_{j \geq i+1} \mathscr{L}_{a_j}$. Therefore, $P = \left(\bigcup_{1 \leq i < k} \bigcap_{j \geq i} \mathscr{L}_{a_j}\right) \subseteq M$. But $P \cup M = N$, implying $ M = N$. So, $\sup(M) = \sup(N)$. We now prove (2). Let $S = \left ( \bigcup_{i \geq k} \bigcap_{j \geq i} \mathscr{R}_{a_j} \right )$ and $T = \left ( \bigcup_{i \geq 1} \bigcap_{j \geq i} \mathscr{R}_{a_j} \right )$. Note that $\bigcap_{j \geq i} \mathscr{R}_{a_j} \subseteq \bigcap_{j \geq i+1} \mathscr{R}_{a_j}$. Therefore, $U = \left(\bigcup_{1 \leq i < k} \bigcap_{j \geq i} \mathscr{R}_{a_j}\right) \subseteq S$. But $U \cup S = T$, implying $ S = T$. So $\inf(S) = \inf(T)$. Statements (1) and (2) suffice to show that $\ell(\mathfrak{B}) = \ell(\mathfrak{A})$.
\end{proof}

\begin{remark}
The analogue of Lemma~\ref{seqs-lem-1} also holds on $\mathbb{R}$ and supports our intuition that the first terms of a sequence have no bearing on the limit of that sequence.
\end{remark}

\begin{theorem} \label{seqs-thm-1}
Let $\mathfrak{A} = a_1, a_2, \dots$ be an $\mathbf{On}$-length sequence. If $\lim_{i \rightarrow \mathbf{On}} a_i = \ell(\mathfrak{A}) \in \mathbf{No}$, then for every (surreal) $\varepsilon > 0$, there is an $N \in \mathbf{On}$ satisfying $\left| a_n-\ell(\mathfrak{A})\right| <\varepsilon$ whenever $n\in \mathbf{On}_{>N}$. Conversely, if $\ell$ is a number such that for every (surreal) $\varepsilon > 0$, there is an $N \in \mathbf{On}$ satisfying $\left| a_n-\ell\right| <\varepsilon$ whenever $n\in \mathbf{On}_{>N}$, then $\lim_{i \rightarrow \mathbf{On}} a_i = \ell$.
\end{theorem}

\begin{proof}
For the forward direction, we must prove that for every $\varepsilon > 0$, there exists $N \in \mathbf{On}$ such that whenever $n\in \mathbf{On}_{>N}$, $|a_n-\ell(\mathfrak{A})|<\varepsilon$. Split $\mathfrak{A}$ into two subsequences, $\mathfrak{A}_+ = b_1, b_2, \dots$ being the subsequence of all terms $\geq \ell(\mathfrak{A})$ and $\mathfrak{A}_- = c_1, c_2, \dots$ being the subsequence of all terms $\leq \ell(\mathfrak{A})$.
Note that the limit of an $\mathbf{On}$-length subsequence equals the limit of its parent sequence.
If either $\mathfrak{A}_+$ or $\mathfrak{A}_-$ has ordinal length (they cannot both be of ordinal length because $\mathfrak{A}$ has length $\mathbf{On}$), by Lemma~\ref{seqs-lem-1}, we can redefine $\mathfrak{A} \defeq a_{\beta}, a_{\beta +1}, \dots$ for some $\beta \in \mathbf{On}$ such that the tail of the new sequence $\mathfrak{A}$ lies entirely in either $\mathfrak{A}_+$ or $\mathfrak{A}_-$, depending on which subsequence has length $\mathbf{On}$.
Let us assume that both $\mathfrak{A}_+, \mathfrak{A}_-$ have length $\mathbf{On}$.

Observe $\left|b_n-\ell(\mathfrak{A})\right| = b_n-\ell(\mathfrak{A})$. Suppose there does not exist $N_1 \in \mathbf{On}$ such that whenever $n\in \mathbf{On}_{>N_1}$, $b_n-\ell(\mathfrak{A})<\varepsilon$ for some $\varepsilon > 0$. Then for arbitrarily many $n>N_1$, $b_n \geq \ell(\mathfrak{A})+\varepsilon$. Thus, $y = \inf\left(\bigcup_{i \geq 1} \bigcap_{j \geq i} \mathscr{R}_{b_i}\right) \geq \ell(\mathfrak{A})+\varepsilon$, a contradiction because $y = \ell(\mathfrak{A})$ if the expression on the right-hand-side of (\ref{seqs-eqn-1}) is the Dedekind representation of $\ell(\mathfrak{A})$. Therefore, there exists $N_1 \in \mathbf{On}$ such that whenever $n>N_1$, $b_n-\ell(\mathfrak{A})<\varepsilon$. A similar argument shows that there exists $N_2 \in \mathbf{On}$ such that whenever $n \in \mathbf{On}_{>N_2}$, $c_n - \ell(\mathfrak{A}) < \varepsilon$.
Then $N = \max\{N_1, N_2\}$ satisfies Definition~\ref{defs-def-5}.

If $\mathfrak{A}_+$ is of ordinal length, then instead of $N = \max\{N_1, N_2\}$ we have $N = N_2$. Similarly, if $\mathfrak{A}_-$ is of ordinal length, then instead of $N = \max\{N_1, N_2\}$ we have $N = N_1$.

For the other direction, if $\lim_{i \rightarrow \mathbf{On}} a_i \neq \ell$, then there are two cases to consider. The first case is that the expression on the right-hand-side of (\ref{seqs-eqn-1}) in Definition~\ref{def15} is not a Dedekind representation. This would imply that
\begin{equation}\label{extraproof}
\inf \left ( \bigcup_{i \geq 1} \bigcap_{j \geq i} \mathscr{R}_{a_j} \right ) - \sup \left ( \bigcup_{i \geq 1} \bigcap_{j \geq i} \mathscr{L}_{a_j} \right )  >\varepsilon,
\end{equation}
for some $\varepsilon > 0$, because otherwise we would have elements of the right class of a number smaller than elements of the left class. But since the $a_i$ can be made arbitrarily close to $\ell$ by taking $i$ sufficiently large, we can pick $x \in \bigcup_{i \geq 1} \bigcap_{j \geq i} \mathscr{R}_{a_j}$ and $y \in \bigcup_{i \geq 1} \bigcap_{j \geq i} \mathscr{L}_{a_j}$ such that $\left|x-\ell\right|<\varepsilon/2$ and $\left|y-\ell\right|<\varepsilon/2$. By the Triangle Inequality, $\left|x-y\right| \leq \left|x-\ell\right| + \left|y-\ell\right| < \varepsilon/2 + \varepsilon/2 < \varepsilon$, which contradicts the claim in (\ref{extraproof}). Thus, it follows that the expression on the right-hand-side of (\ref{seqs-eqn-1}) in Definition~\ref{def15} is a Dedekind representation.

The second case is that $\ell(\mathfrak{A})$ is a gap, not a number. Since the expression on the right-hand-side of (\ref{seqs-eqn-1}) in Definition~\ref{def15} is a Dedekind representation, we have that
\begin{equation} \label{extraproof2}
\inf \left ( \bigcup_{i \geq 1} \bigcap_{j \geq i} \mathscr{R}_{a_j} \right ) = \sup \left ( \bigcup_{i \geq 1} \bigcap_{j \geq i} \mathscr{L}_{a_j} \right ) = \ell(\mathfrak{A}).
\end{equation}
Now suppose that $\ell(\mathfrak{A}) < \ell$. Using notation from the proof of the first case above, we know that for every $\varepsilon > 0$, we have $\left|y-\ell\right|<\varepsilon$. If we pick $\varepsilon$ such that $\ell -\varepsilon > \ell(\mathfrak{A})$, then we have $y > \ell(\mathfrak{A})$, which contradicts (\ref{extraproof2}). Thus, $\ell(\mathfrak{A}) \not< \ell$. By analogous reasoning in which we replace ``left" with ``right," we find that $\ell(\mathfrak{A}) \not> \ell$. Finally, we have $\ell(\mathfrak{A}) = \ell$, so $\ell(\mathfrak{A})$ cannot be a gap. This completes the proof of the theorem.
\end{proof}

It is only natural to wonder why in our statement of Theorem~\ref{seqs-thm-1} we restrict our consideration to sequences approaching numbers. The issue with extending the theorem to describe gaps is nicely demonstrated in the example sequence $\mathfrak{A} = a_1, a_2, \dots$ defined by $a_i = \omega^{1/i}$. Substituting $\mathfrak{A}$ into the expression on the right-hand-side of (\ref{seqs-eqn-1}) in Definition~\ref{def15} yields $\ell(\mathfrak{A}) = \infty$. However, it is clearly not true that we can make $a_i$ arbitrarily close to $\infty$ by picking $i$ sufficiently large, because for every surreal $\varepsilon \in (0,1)$ and every $i \in \mathbf{On}$, we have that $a_i -\varepsilon > \infty$.

We next consider how our method of evaluating limits of sequences using Dedekind representations can be employed to completely characterize convergent sequences.

\subsection{Cauchy Sequences}\label{cauchyseq}
We can now distinguish between sequences that converge (to numbers), sequences that approach gaps, and sequences that neither converge nor approach gaps. On the real numbers, all Cauchy sequences converge; i.e.~$\mathbb{R}$ is Cauchy complete. However, $\mathbf{No}$ is not Cauchy complete; there are Cauchy sequences of numbers that approach gaps. We devote this subsection to determining what types of Cauchy sequences converge (to numbers) and what types do not. Let us begin our formal discussion of Cauchy sequences by defining them as follows:
\begin{defn} \label{defcauchy}
Let $\mathfrak{A} = a_1, a_2, \dots$ be a sequence of length $\mathbf{On}$. Then $\mathfrak{A}$ is a Cauchy sequence if for every (surreal) $\varepsilon > 0$ there exists $N \in \mathbf{On}$ such that whenever $m, n \in \mathbf{On}_{>N}$, $\left| a_m - a_n \right| < \varepsilon$.
\end{defn}

As follows, we show that $\mathbf{No}$ is not Cauchy complete by providing an example of a sequence that satisfies Definition~\ref{defcauchy} but approaches a gap.

\begin{example} \label{5.4}
Let $\mathfrak{A} = 1, 1 + 1/\omega, 1 + 1/\omega + 1/\omega^2, 1 + 1/\omega + 1/\omega^2+1/\omega^3, \dots$. Note that $\mathfrak{A}$ is a Cauchy sequence because it satisfies Definition~\ref{defcauchy}; i.e.~for every $\varepsilon > 0$, there exists $N \in \mathbf{On}$ such that whenever $m, n \in \mathbf{On}_{>N}$, $\left|\sum_{i \in \mathbf{On}_{\leq m}} 1/\omega^i - \sum_{i \in \mathbf{On}_{\leq n}} 1/\omega^i \right|< \varepsilon$. It is easy to check that the Dedekind representation of $\ell(\mathfrak{A})$ in this example, it is easy to check that the Dedekind representation obtained for  $\ell(\mathfrak{A})$ is that of the object $\sum_{i \in \mathbf{On}} 1/\omega^i$. However, as explained in Subsection~\ref{gapsadd}, $\sum_{i \in \mathbf{On}} 1/\omega^i$ is the normal form of a gap, so $\mathfrak{A}$ is a Cauchy sequence that approaches a gap, a result that confirms the fact that $\mathbf{No}$ is not Cauchy complete.
\end{example}

Four key steps make up our strategy for classifying Cauchy sequences: (1) first prove that sequences approaching Type II gaps are not Cauchy; (2) second conclude that Cauchy sequences either converge, approach Type I gaps, or diverge (as it happens, Cauchy sequences do not diverge, but we only prove this in step 4); (3) third prove that only a certain kind of Type I gap can be approached by Cauchy sequences; and (4) fourth prove that Cauchy sequences that do not approach such Type I gaps are convergent. We execute this strategy as follows.

To prove that sequences approaching Type II gaps are not Cauchy, we need a restriction on the definition of gaps. We restrict Conway's original definition of gaps as follows:

\begin{defn} \label{def-nums-2}
A surreal gap is any Dedekind section of $\mathbf{No}$ that cannot be represented as either $\{\mathbf{No}_{<x} \mid  \mathbf{No}_{\geq x}\}$ or $\{\mathbf{No}_{\leq x} \mid  \mathbf{No}_{> x}\}$ for some $x \in \mathbf{No}$. Furthermore, objects of the form $\{\mathbf{No}_{<x} \mid  \mathbf{No}_{\geq x}\}$ or $\{\mathbf{No}_{\leq x} \mid  \mathbf{No}_{> x}\}$ are defined to be equal to $x$.
\end{defn}

\begin{remark}
From now on, the unqualified word ``gap'' refers only to gaps of the type described in Definition~\ref{def-nums-2}.
\end{remark}

\begin{lemma} \label{Cauchylem}
Let $\mathfrak{A} = a_1, a_2, \dots$. If $\lim_{i \rightarrow \mathbf{On}} a_i = g$ for some gap $g$ of Type II, then the sequence $\mathfrak{A}$ is not Cauchy.
\end{lemma}
\begin{proof}
Suppose $\mathfrak{A}$ is Cauchy. It follows that $\left|a_i^\mathscr{R}-a_j^\mathscr{L}\right|$ can be made arbitrarily close to $0$ if $i, j \in \mathbf{On}$ are taken sufficiently large. Then, if $\ell(\mathfrak{A})$ denotes the Dedekind representation of $g$, $\left|\ell(\mathfrak{A})^\mathscr{R}-\ell(\mathfrak{A})^\mathscr{L}\right|$ can be made arbitrarily close to $0$. It follows that $\left|\ell(\mathfrak{A})^\mathscr{R}-g\right|$ can be made arbitrarily close to $0$ if $i, j \in \mathbf{On}$ are taken sufficiently large. As discussed in the previous section, we know that $g = \sum_{i \in \mathbf{On}_{< \beta}}r_{i}\omega^{y_{i}}\oplus\left(\pm \omega^{\Theta}\right)$ for some gap $\Theta$. Also, as described in Subsection~\ref{defsnums}, $h = \sum_{i \in \mathbf{On}_{< \beta}}r_{i}\omega^{y_{i}}$ is a number, so $\left|g-\ell(\mathfrak{A})^\mathscr{R}\right| = \left|h' \oplus\left(\pm \omega^{\Theta}\right)\right|$, where $h' \in \mathbf{No}$. Now, $\left|h' \oplus\left(\pm \omega^{\Theta}\right)\right|$ is a gap, and $\Theta > \mathbf{Off}$ because $\omega^\mathbf{Off} = 1/\mathbf{On}$ is not a gap by Definition~\ref{def-nums-2}. Because we can make $\left|h' \oplus\left(\pm \omega^{\Theta}\right)\right|$ smaller than any (surreal) $\varepsilon > 0$, pick $\varepsilon = \omega^r$ for some number $r<\Theta$ (which is possible to do because $\Theta > \mathbf{Off}$). Then we can either have (1) $h' > \omega^\Theta > h'-\omega^r$; or (2) $h' < \omega^\Theta < h' + \omega^r$. In case (1) let $z$ denote the largest power of $\omega$ in the normal form of $h'$. Clearly, $z>\Theta$ and $z > r$, so the largest power of $\omega$ in the normal form of $h'-\omega^r$ is $z$. But it follows that $h'-\omega^r > \omega^\Theta$, a contradiction. Similarly, in case (2) let $z$ denote the largest power of $\omega$ in the normal form of $h'$. Clearly, $z<\Theta$, so the largest power of $\omega$ in the normal form of $h'+\omega^r$ is $\max\{z,r\}$. But it follows that $h'+\omega^r < \omega^\Theta$, a contradiction. Thus we have the lemma.
\end{proof}

It might seem like the gap restriction of Definition~\ref{def-nums-2} was imposed as a convenient means of allowing Lemma~\ref{Cauchylem} to hold. Nevertheless, there is sound intuitive reasoning for why we must restrict gaps in this way. In Definition~\ref{defs-def-4}, gaps are defined to be Dedekind sections of $\mathbf{No}$. This means that sections like $1/\mathbf{On} = \{\mathbf{No}_{\leq 0} \mid  \mathbf{No}_{ > 0}\}$ are gaps. But if we allow objects such as $1/\mathbf{On}$ to be gaps, we can create similar ``gaps'' in the real line by claiming that there exist: (1) for each $a \in \mathbb{R}$, an object $>a$ and less than all reals $>a$; and (2) another object $<a$ and greater than all reals $< a$. However, such objects are not considered to be ``gaps'' in the real line. Additionally, without Definition~\ref{def-nums-2}, $\mathbf{On}$-length sequences like $1, 1/2, 1/4, \dots$ would be said to approach the gap $1/\mathbf{On}$ rather than the desired limit $0$, and in fact no $\mathbf{On}$-length sequences would converge at all. For these reasons, we must restrict the definition of gaps.

Now, Cauchy sequences must therefore either converge, approach Type I gaps, or diverge. Let us next consider the case of Cauchy sequences approaching Type I gaps. Designate a Type I gap $g_1 = \sum_{i \in \mathbf{On}} r_i \cdot \omega^{y_i}$ to be a \emph{Type Ia} gap iff $\lim_{i \rightarrow \mathbf{On}} y_i = \mathbf{Off}$ and to be a \emph{Type Ib} gap otherwise. We now state and prove the following lemma about Type I gaps:

\begin{lemma} \label{lem2}
Let $\mathfrak{A} = a_1, a_2, \dots$ be a Cauchy sequence. If $\lim_{i \rightarrow \mathbf{On}} a_i = g$ for some gap $g$ of Type I, then $g$ is a gap of Type Ia.
\end{lemma}

\begin{proof}
Suppose $g = \sum_{i \in \mathbf{On}} r_i \cdot \omega^{y_i}$ is of Type Ib. Because the $y_i$ are a decreasing sequence that does not approach $\mathbf{Off}$, they are bounded below by some number, say $b$. Now since $\mathfrak{A}$ is Cauchy, it follows that $\left|a_i^\mathscr{R}-a_j^\mathscr{L}\right|$ can be made arbitrarily close to $0$ if $i, j \in \mathbf{On}$ are taken sufficiently large. Then, if $\ell(\mathfrak{A})$ denotes the Dedekind representation of $g$, $\left|\ell(\mathfrak{A})^\mathscr{R}-\ell(\mathfrak{A})^\mathscr{L}\right|$ can be made arbitrarily close to $0$. It follows that $\left|\ell(\mathfrak{A})^\mathscr{R}-g\right|$ can be made arbitrarily close to $0$ if $i, j \in \mathbf{On}$ are taken sufficiently large. In particular, we can choose $\ell(\mathfrak{A})^\mathscr{R}$ so that $\left|\ell(\mathfrak{A})^\mathscr{R}-g\right|<\omega^b$. Then, either (1) $\ell(\mathfrak{A})^\mathscr{R} > g > \ell(\mathfrak{A})^\mathscr{R} - \omega^b$; or (2) $\ell(\mathfrak{A})^\mathscr{R} < g < \ell(\mathfrak{A})^\mathscr{R} + \omega^b$. In case (1), the largest exponent $z$ of $\omega$ in the normal form of the (positive) object $\ell(\mathfrak{A})^\mathscr{R}-g$ satisfies $z \geq y_\alpha$ for some $\alpha \in \mathbf{On}$. Therefore, $z > b$, so clearly $\ell(\mathfrak{A})^\mathscr{R}-g>\omega^b$, a contradiction. In case (2), the largest exponent $z$ of $\omega$ in the normal form of the (positive) object $g-\ell(\mathfrak{A})^\mathscr{R}$ satisfies $z = y_\alpha$ for some $\alpha \in \mathbf{On}$. Therefore, $z > b$, so clearly $g-\ell(\mathfrak{A})^\mathscr{R}>\omega^b$, a contradiction. Thus, we have the lemma.
\end{proof}

\begin{remark}
Consider the case of a Cauchy sequence $\mathfrak{A} = a_1, a_2, \dots$ that approaches a Type Ia gap $g$. Note that $\mathfrak{A}$ approaches $g$ iff it is equivalent to the sequence $\mathfrak{B}$ defined by successive partial sums of the normal form of $g$. Here, two sequences $\{a_n\}$ and $\{b_n\}$ are considered \emph{equivalent} iff for every (surreal) $\varepsilon > 0$ there exists $N \in \mathbf{On}$ such that whenever $n>N$, $\left|a_n-b_n\right| < \varepsilon$.
\end{remark}

From Lemmas~\ref{Cauchylem} and~\ref{lem2}, we know that there is a Cauchy sequence that approaches a gap iff the gap is of Type Ia (the reverse direction follows easily from the properties of the normal form of a Type Ia gap). We now prove that Cauchy sequences that do not approach gaps of Type Ia must converge (to numbers).

\begin{theorem} \label{seqs-thm-2}
Let $\mathfrak{A} = a_1, a_2, \dots$ be a Cauchy sequence that does not approach a gap of Type Ia. Then $\lim_{i \rightarrow \mathbf{On}} a_i \in \mathbf{No}$.
\end{theorem}
\begin{proof}
We first prove that $\mathfrak{A}$ is bounded. Let $\alpha \in \mathbf{On}$ such that for all $\beta, \gamma \in \mathbf{On}_{\geq\alpha}$, we have that $\left|a_\beta - a_\gamma\right|<1$. Then for all $\beta \geq \alpha$, by the Triangle Inequality we have that $\left|a_\beta\right| < \left|a_\alpha\right| + 1$. So, for all $\beta \in \mathbf{On}$, we have that $\left|a_\beta\right| \leq \max\{\left|a_1\right|,\left|a_2\right|,\dots, \left|a_\alpha\right| + 1\}$. Thus $\mathfrak{A}$ is bounded.

Now consider the class $C = \{x\in\mathbf{No} : x<a_\alpha$ for all $\alpha$ except ordinal-many$\}$. We next prove that $\sup(C)\in\mathbf{No}$. If not, then $\sup(C)$ is a gap, say $g$, and $g\neq\mathbf{On}$ since $\mathfrak{A}$ is bounded. Because $\mathfrak{A}$ is Cauchy, we claim that $D = \{x\in\mathbf{No} : x>a_\alpha$ for all $\alpha$ except ordinal-many$\}$ satisfies $\inf(D)=g$. If this claim were untrue, then we can find two numbers $p,q \in (\sup(C), \inf(D))$ (there are at least two numbers in this interval because of the gap restriction of Definition~\ref{def-nums-2}) so that $\left|a_\beta-a_\gamma\right| > |p-q|$ for $\mathbf{On}$-many $\beta, \gamma \in \mathbf{On}$, which contradicts the fact that $\mathfrak{A}$ is Cauchy. Thus the claim holds. Now  $\mathfrak{A}$ satisfies Definition~\ref{defs-def-5}, so $\lim_{i \rightarrow \mathbf{On}} a_i = g$, where by assumption $g$ must be a gap of Type II or a gap of Type Ib. But, by Lemmas~\ref{Cauchylem} and~\ref{lem2}, $\mathfrak{A}$ cannot be Cauchy, a contradiction. So $\sup(C) = \inf(D) \in\mathbf{No}$.

We finally prove that for every $\varepsilon > 0$, we can find $\alpha \in \mathbf{On}$ so that for every $\beta \in \mathbf{On}_{>\alpha}$, we have $\left|a_\beta-\sup(C)\right|<\varepsilon$. Suppose the contrary, so that for some $\varepsilon > 0$ we have (1) $a_\beta \leq \sup(C) - \varepsilon$ for $\mathbf{On}$-many $\beta$; or (2) $a_\beta \geq \sup(C) + \varepsilon$ for $\mathbf{On}$-many $\beta$. In the first case, we find that there is an upper bound of $C$ that is less than $\sup(C)$, which is a contradiction, and in the second case, we find that there is a lower bound of $D$ that his greater than $\inf(D)$, which is again a contradiction. Thus we have the theorem.
\end{proof}

\section{Limits of Functions and Intermediate Value Theorem} \label{diff}

\noindent In this section, we present a Dedekind representation for the limit of a function and prove the Intermediate Value Theorem.

\subsection{Evaluation of Limits of Functions} \label{limitfunceval}

Finding a genetic formula for the limit of a function is a task that faces issues similar to those described in Subsection~\ref{priorat}. Suppose $f$ is a function whose domain is $\mathbf{No}$. Then, by Definition~\ref{defs-def-6}, there exists $\delta>0$ such that for some $\beta \in \mathbf{On}$ and for all $\varepsilon \in (0, 1/\omega^{\beta})$, $\lim_{x \rightarrow a} f(x)-\varepsilon < f(x) < \lim_{x \rightarrow a} f(x)+\varepsilon$ whenever $| x-a|  < \delta$. Thus, we can make $\mathfrak{b}(\delta)$ arbitrarily large by taking $\beta$ sufficiently large, so any reasonable genetic formula $\{L \mid R\}$ for $\lim_{x \rightarrow a} f(x)$ would depend on values of $x$ of arbitrarily large birthday. Thus, $L$ and $R$ would again be too large to be sets. Because differentiation is taking the limit of the function $\frac{f(x+h)-f(x)}{h}$ as $h \rightarrow 0$, we also cannot differentiate functions while still satisfying Definition~\ref{defs-def-1}. We conclude that a different representation of numbers (namely the Dedekind representation) that allows $L$ and $R$ to be proper classes is necessary for limits and derivatives of functions to be defined for surreals.

The following is our definition of the limit of a surreal function. Notice that the definition is analogous to that of the limit of an $\mathbf{On}$-length surreal sequence, Definition~\ref{defs-def-5}.

\begin{defn} \label{defs-def-6}
Let $f$ be a function defined on an open interval containing $a$, except possibly at $a$. We say that $f(x)$ converges to a limit $\ell$ as $x \rightarrow a$ and write that $\lim_{x \rightarrow a} f(x) = \ell$ if the expression in (\ref{diff-eqn-1}) is the Dedekind representation of $\ell$.
\begin{equation} \label{diff-eqn-1}
\left \{ p : p < \sup \left ( \bigcup_{b < x < a} \,\,\, \bigcap_{x \leq y < a} \mathscr{L}_{f(y)} \right ) \middle\rvert  q : q > \inf \left ( \bigcup_{a < x < c} \,\,\, \bigcap_{a < y \leq x} \mathscr{R}_{f(y)} \right ) \right \}
\end{equation}
\end{defn}

\begin{remark}
Definition~\ref{defs-def-6} holds for all surreal functions $f$ defined on an open interval containing $a$, except possibly at $a$. In particular, $f$ need not have a genetic definition.
\end{remark}

The Dedekind representation notion of the limit of a surreal function is equivalent to the standard $\varepsilon$-$\delta$ definition.

\begin{theorem} \label{diff-thm-1}
Let $f$ be defined on $(b,c) \subsetneq \mathbf{No}$ containing $a$, except possibly at $a$. If $\lim_{y \rightarrow a} f(y) = \ell \in \mathbf{No}$, then for every (surreal) $\varepsilon > 0$ there exists $\delta > 0$ such that whenever $0 < \left| y-a\right| <\delta$, $\left| f(y)-\ell\right| <\varepsilon$. Conversely, if $\ell$ is a number such that for every (surreal) $\varepsilon > 0$ there exists $\delta > 0$ such that whenever $0 < \left| y-a\right| <\delta$, $\left| f(y)-\ell\right| <\varepsilon$, then $\lim_{y \rightarrow a} f(y) = \ell$.
\end{theorem}

\begin{proof}
The proof is analogous to that of Theorem~\ref{seqs-thm-1}, so we omit it.
\end{proof}

\begin{remark}
As with limits of $\mathbf{On}$-length sequences, we restrict Theorem~\ref{diff-thm-1} to functions that converge to numbers because the standard $\varepsilon$-$\delta$ definition does not generalize to gaps. Also, notice that the expression on the right-hand-side of (\ref{seqs-eqn-1}) in Definition~\ref{def15} as well as Definition~\ref{defs-def-5} can be easily modified to provide a definition of a limit of a function $f(x)$ as $x \rightarrow \mathbf{On}$ or $x \rightarrow \mathbf{Off}$. Finally, the notions of $\lim_{x \rightarrow a^-} f(x)$ and $\lim_{x \rightarrow a^+} f(x)$ are also preserved in Theorem~\ref{diff-thm-1}. Specifically, the left class of $\lim_{x \rightarrow a} f(x)$ describes the behavior of $f(x)$ as $x \rightarrow a^-$ and the right class of $\lim_{x \rightarrow a} f(x)$ describes the behavior of $f(x)$ as $x \rightarrow a^+$.
\end{remark}
Notice that derivatives are limits of functions; i.e.~$\frac{d}{dx} f(x)  = \lim_{h \rightarrow 0} g(h)$, where $g(h) = \frac{f(x+h) - f(x)}{h}$. Therefore, derivatives can be evaluated using Definition~\ref{defs-def-6}. Evaluating limits and derivatives of functions can be made easier through the use of limit laws. We introduce two limit laws in the following proposition:

\begin{proposition} \label{diff-thm-2}
Let $a \in \mathbf{No}$ and $f, g$ be functions, and suppose that $\lim_{x \rightarrow a} f(x)$ and $\lim_{x \rightarrow a} g(x)$ both exist. Then, the following hold: (1) Addition: $\lim_{x \rightarrow a} (f+g)(x) = \lim_{x \rightarrow a} f(x) + \lim_{x \rightarrow a} g(x)$; and (2) Multiplication: $\lim_{x \rightarrow a} (f \cdot g)(x) = \lim_{x \rightarrow a} f(x) \cdot \lim_{x \rightarrow a} g(x).$
\end{proposition}
\begin{proof}
The proof is a routine calculation, so we omit it. One can either use the standard $\varepsilon$-$\delta$ arguments or the Dedekind representation of the limit of a function; because these two notions of limit are equivalent, either method will work.
\end{proof}

The notion of limits of functions gives rise to a weaker version of continuity, which is defined as follows:
\begin{defn}\label{weakcont}
Let $f : A \to \mathbf{No}$ be a function defined on a locally open class $A$. Then $f$ is \emph{weakly continuous} at $a \in A$ if $\lim_{x \rightarrow a} f(x) = f(a)$.
\end{defn}

We now justify the terminology \emph{weakly} continuous:

\begin{proposition}\label{contprop}
Let $f : A \to \mathbf{No}$ be a continuous function defined on a locally open class $A$. Then $f$ is weakly continuous on $A$.
\end{proposition}
\begin{proof}
Let $a \in A, \varepsilon > 0$. Then $f^{-1}((f(a)-\varepsilon, f(a) + \varepsilon))$ is open in $A$ (under the topology of Definition~\ref{defopen}), because $(f(a)-\varepsilon, f(a) + \varepsilon)$ is open (under the same topology). There exists $\delta > 0$ such that $(a-\delta, a+\delta) \subset f^{-1}((f(a)-\varepsilon, f(a) + \varepsilon))$. Then $f((a-\delta, a+\delta)) \subset (f(a)-\varepsilon, f(a) + \varepsilon)$. It follows that $f$ is weakly continuous on $A$.
\end{proof}

\subsection{Intermediate Value Theorem} \label{IVTsection}
Even though the standard proof of the Intermediate Value Theorem (IVT) on $\mathbb{R}$ requires completeness, we show in this subsection that we can prove the IVT for surreals without using completeness by using Definitions~\ref{defopen} and~\ref{defcont} along with the following results regarding connectedness in $\mathbf{No}$. The proofs below are surreal versions of the corresponding proofs from Munkres' textbook~\cite{Mun75}.

\begin{defn}
A class $T \subset \mathbf{No}$ is connected if there does not exist a separation of $T$; i.e.~there does not exist a pair of disjoint nonempty classes $U,V $ that are open in $T$ such that $T = U \cup V$.
\end{defn}

\begin{lemma}\label{intcon}
Every convex class $T \subset \mathbf{No}$ is connected.
\end{lemma}

\begin{proof}
Suppose the pair of classes $U,V$ forms a separation of $T$. Then, take $u \in U$ and $v \in V$, and assume without loss of generality that $u < v$ ($u \neq v$ because $U \cap V = \varnothing$). Because $T$ is convex, we have that $[u,v] \subset T$, so consider the pair of classes $U' = U \cap [u,v], V' = V \cap [u,v]$. Notice that (1) because $U \cap V = \varnothing$, we have that $U' \cap V' = \varnothing$; (2) $u \in U'$ and $v \in V'$, so neither $U'$ nor $V'$ is empty; (3) $U'$ and $V'$ are open in $[u,v]$; and (4) clearly $U' \cup V' = [u,v]$. So, the pair $U',V'$ forms a separation of $[u,v]$.

Now consider $w = \sup(U')$. If $w \in \mathbf{No}$, then we have two cases: (1) $w \in V'$ and (2) $w \in U'$. In case (1), because $V'$ is open, there is an interval contained in $V'$ of the form $(x, w]$, for some number or gap $x$. Then $x$ is an upper bound of $U'$ that is less than $w$, because all numbers between $w$ and $v$ inclusive are not in $U'$, which contradicts the definition of $w$. In case (2), because $U'$ is open, there is an interval contained in $U'$ of the form $[w, y)$, for some number or gap $y$. But then any $z \in [w,y)$ satisfies both $z \in U'$ and $z > w$, which again contradicts the definition of $w$.

If $w$ is a gap, then there is no $x \in V'$ such that both $x < w$ and $(x,w) \subset V'$ hold, because then $w \neq \sup(U')$. Thus, the class $\mathcal{V} = \{x \in V' : x > w\}$ is open, for no intervals contained in $V'$ lie across $w$. Now, $w = \{U' \mid \mathcal{V}\}$, but because $U'$ and $\mathcal{V}$ are open, they are unions of intervals indexed over \emph{sets}. Consequently, $w = \{L \mid R\}$ for some pair of sets $L,R$, so $w \in \mathbf{No}$, which contradicts our assumption that\mbox{ $w$ is a gap.}
\end{proof}

\begin{lemma}\label{confunc}
If $f$ is continuous on $[a,b]$, then the image $f([a,b])$ is connected.
\end{lemma}

\begin{proof}
Suppose $f([a,b])$ is not connected. Then there exists a separation $U,V$ of $f([a,b])$; i.e. there exists a pair of disjoint nonempty open classes $U,V$ such that $f([a,b]) = U \cup V$. Then, we have the following: (1) $f^{-1}(U)$ and $f^{-1}(V)$ are disjoint because $U,V$ are disjoint; (2) $f^{-1}(U)$ and $f^{-1}(V)$ are nonempty because the map from $[a,b]$ to the image $f([a,b])$ under $f$ is clearly surjective; (3) $f^{-1}(U)$ and $f^{-1}(V)$ are open in $[a,b]$ because $f$ is continuous, so the preimages of open classes are open in $[a,b]$; and (4) $[a,b] = f^{-1}(U) \cup f^{-1}(V)$ because any point whose image is in $U$ or $V$ must be in the corresponding preimage $f^{-1}(U)$ or $f^{-1}(V)$. Thus, the pair $f^{-1}(U), f^{-1}(V)$ forms a separation of $[a,b]$. But this contradicts Lemma~\ref{intcon}, because intervals are convex, so $f([a,b])$ is connected.
\end{proof}

\begin{theorem}[IVT]
If $f$ is continuous on $[a,b] \subset \mathbf{No}$, then for every $u \in \mathbf{No}$ that lies between $f(a)$ and $f(b)$, there exists a number $p \in [a,b]$ such that $f(p) = u$.
\end{theorem}

\begin{proof}
Assume that neither $f(a) = u$ nor $f(b) = u$ (if either of these were true, we would have the theorem). Consider the classes $U = f([a,b]) \cap (\mathbf{Off}, u)$ and $V = f([a,b]) \cap (u, \mathbf{On})$. Notice that (1) $U \cap V = \varnothing$ because $(\mathbf{Off}, u) \cap (u, \mathbf{On}) = \varnothing$; (2) neither $U$ nor $V$ is empty because either $f(a) < f(b)$ so $f(a) \in U$ and $f(b) \in V$, or $f(a) > f(b)$ so $f(a) \in V$ and $f(b) \in U$; and (3) both $U$ and $V$ are open in $f([a,b])$ (but not necessarily in $\mathbf{No}$) because each is the intersection of $f([a,b])$ with an open ray. Now assume there is no $p \in [a,b]$ such that $f(p) = u$. Because $f([a,b]) = U \cup V$, we have that the pair $U,V$ is a separation for $f([a,b])$, so $f([a,b])$ is not connected. But this violates Lemma~\ref{confunc}, so there is a $p \in [a,b]$ such that $f(p) = u$.
\end{proof}

That the IVT holds for surreals does not of course prevent functions from reaching numbers at gaps, but it does prevent continuous functions from having isolated zeroes at gaps.\footnote{The zero-function reaches $0$ at every gap.} More precisely, a continuous function $f$ can reach $0$ at a gap $g$ iff for every (surreal) $\varepsilon > 0$ there exists a zero of $f$ in some open interval of width $\varepsilon$ containing $g$. The only continuous functions we know of that reach a number at a gap are constant on an open interval containing the gap, but we do not yet know whether it is impossible for a continuous function to reach a number at a gap without being locally constant.

\section{Series and Integrals} \label{ips}

\noindent In this section, we present our methods of evaluating series and infinite ``Riemann'' sums. We also prove the Fundamental Theorem of Calculus as a method of evaluating integrals easily, as long as we have a definition of integration. A consistent genetic definition or Dedekind representation of Riemann integration that works for all functions nevertheless remains to be discovered.

\subsection{Series} \label{powersection}
The evaluation of series in real analysis usually entails finding a limit of a sequence of partial sums. Because Definition~\ref{defs-def-5} allows us to find limits of sequences, it might seem as though evaluating series as limits of partial sum sequences is possible. However, it is often the case that we do not know a closed-form expression for the $\alpha^\mathrm{th}$ partial sum of a series, where $\alpha \in \mathbf{On}$. Without such an expression, we cannot determine the left and right options of the $\alpha^\mathrm{th}$ partial sum and therefore cannot use Definition~\ref{defs-def-5}. Also, suppose $\zeta$ is a limit-ordinal. Then by Theorem~\ref{thm-priorat-1}, we cannot claim that the $\zeta^\mathrm{th}$ partial sum is the limit of previous partial sums, for this limit might be a gap. The next example illustrates how the partial sums of an $\mathbf{On}$-length series can become ``stuck" at a gap:

\begin{example} \label{ips-examp-1}
Consider the series $s = \sum_{i \in \mathbf{On}} 1/2^i$. The sum of the first $\omega$ terms, $\sum_{i \in \mathbf{On}_{<\omega}} 1/2^i$, is the gap $g$ between numbers with real part less than $2$ and numbers with real part  at least $2$. By the definition of gap addition described in Conway's book~\cite{Con01}, it is clear that $g + \varepsilon = g$ for all infinitesimals $\varepsilon$. But note that all remaining terms in the series, namely $1/2^\omega, 1/2^{\omega + 1}, \dots$, are infinitesimals, so the sequence of partial sums for the entire series is: $1, 3/2, 7/4, \dots, g, g, \dots, g, g, \dots$. So, if we define the sum of a series to be the limit of its partial sums, we find that $s = g$, a result that is against our intuition that $s=2$ (which is true for the real series $\sum_{i = 0}^{\infty} 1/2^i$).
\end{example}

Because of the problem described in Example~\ref{ips-examp-1}, we cannot use the standard notion of ``sum'' in order to create surreal series that behave like real series. Our solution to this problem is to extrapolate natural partial sums (which can be evaluated using part 3 of Definition~\ref{defs-def-2}) to ordinal partial sums. We define this method of extrapolation as follows. Denote by $\ol{K}$ the closure of the set of functions $K$ containing rational functions, $\exp(x)$, $\log(x)$, and $\ul{\arctan}(x)$ under the operations of addition, multiplication, and composition. Then, we have the following:

\begin{defn} \label{ips-def-1}
Let $\sum_{i=0}^{\mathbf{On}}a_i$ be a series. Suppose that for all $n \in \mathbb{N}$, $\sum_{i=0}^{n}a_i = f(n)$, where $f \in \ol{K}$. Then for all $\alpha \in \mathbf{On}$, define $\sum_{i=0}^{\mathbf{\alpha}}a_i \defeq f(\alpha)$.
\end{defn}

\begin{remark}
The elements of $\ol{K}$ are what we mean by ``closed-form expressions.'' Notice that closed-form expressions are well-defined.
\end{remark}

Definition~\ref{ips-def-1} is intended to be used with Definition~\ref{defs-def-5}. Specifically, $\sum_{i=0}^{\mathbf{On}}a_i = \lim_{\alpha \rightarrow \mathbf{On}} f(\alpha)$, which is easy to evaluate using  Definition~\ref{defs-def-5}. The rule in Definition~\ref{ips-def-1} does not necessarily work when the $n^{\mathrm{th}}$ partial sum of a series is not an element of $\ol{K}$. If we return to the case in Example~\ref{ips-examp-1}, we see that Definition~\ref{ips-def-1} does indeed yield $\sum_{i\in\mathbf{On}} 1/2^i=2$, as desired. Moreover, we find that $\frac{1}{1-x} = 1+x+x^2+\dots$ holds on the interval $-1 < x < 1$, as usual. However, the power series of other known functions including $e^x,$ $\ul{\arctan}(x),$ and $\mathrm{nlog}(x)$ cannot be evaluated using Definition~\ref{ips-def-1}. Nevertheless, extrapolation seems to be useful for evaluating series, and further investigation might lead to a more general method of evaluating series.

\subsection{Integrals and the Fundamental Theorem of Calculus} \label{integralFTC}
Real integrals are usually defined as limits of Riemann sums. Because Definition~\ref{defs-def-5} gives us the limits of $\mathbf{On}$-length sequences and since we know how to evaluate certain kinds of sums using Definition~\ref{ips-def-1}, we now discuss how we can evaluate certain Riemann sums. We define distance and area, which are necessary for integration. We now define distance in $\mathbf{No}^2$ to be analogous to the distance metric in $\mathbb{R}^2$:
\begin{defn} \label{ips-def-2}
Let $A = (a_1,a_2), B = (b_1, b_2) \in \mathbf{No}^2$. Then the distance $AB$ from $A$ to $B$ is defined to be $AB = \sqrt{(b_1-a_1)^2 + (b_2-a_2)^2}$.
\end{defn}
It is shown in Alling's book~\cite{All87} that distance as defined in Definition~\ref{ips-def-2} satisfies the standard properties of distance that holds in $\mathbb{R}$. Also, in $\mathbb{R}^2$, a notion of the area of a rectangle exists. We define the area of a rectangle in $\mathbf{No}^2$ to be analogous to the area of a rectangle in $\mathbb{R}^2$:

\begin{defn} \label{ips-def-3}
If $ABCD$ is a rectangle in $\mathbf{No}^2$, its area is $[ABCD] = AB \cdot BC$.
\end{defn}

Because we have defined the area of a rectangle in $\mathbf{No}^2$, we can consider Riemann sums. It is easy to visualize a Riemann sum in which the interval of integration is divided into finitely many subintervals. However, when the number of subintervals is allowed to be any ordinal, it is not clear what adding up the areas of an infinite number of rectangles means. For this reason, earlier work has restricted Riemann sums to have only finitely many terms. Just as we did with series earlier, we make an ``extrapolative'' definition for what we want the $\alpha^\mathrm{th}$ Riemann sum of a function to be, thereby defining what we mean by an infinite Riemann sum. Denote by $\ol{K}'$ the closure of the set of functions $K'$ containing rational functions of three variables, $\exp(x)$, $\log(x)$, and $\ul{\arctan}(x)$ under the operations of addition, multiplication, and composition (again, we refer to the elements of $\ol{K}'$ as ``closed-form expressions''). Then, we have the following:

\begin{defn} \label{ips-def-4}
Let $f$ be a function continuous except at finitely many points on $[a,b]$. Suppose that for all $n \in \mathbb{N}$ and $c,d \in [a,b]$ such that $c\leq d$, $g(n,c,d) = \sum_{i = 0}^n \frac{d-c}{n} f\left(c+i\left(\frac{d-c}{n}\right)\right)$, where $g \in \ol{K}'$. Then, for all $\alpha \in \mathbf{On}$, define the $\alpha^{\mathrm{th}}$ Riemann sum of $f$ on $[a,b]$ to be $g(\alpha,a,b)$.
\end{defn}

For functions like polynomials and exponentials, Definition~\ref{ips-def-4} in combination with Definition~\ref{defs-def-5} evaluates integrals correctly. In particular, if $g(\alpha,a,b)$ is known for some function $f(x)$ on the interval $[a,b]$, then $\int_a^b f(x) dx = \lim_{\alpha \rightarrow \mathbf{On}} g(\alpha,a,b)$. We demonstrate this method as follows:

\begin{example} \label{ips-examp-2}
Let us evaluate $\int_a^b \exp(x) dx$ by using our ``extrapolative notion'' of Riemann sums. In this case, we have the following, which results when we use Definition~\ref{ips-def-4}:
\begin{eqnarray*}
g(\alpha,a,b) & = & \frac{b-a}{\alpha} \sum_{i=0}^\alpha \exp\left(a + i\left(\frac{b-a}{\alpha}\right)\right) \\
          & = & \frac{(a-b)\exp(a+a/\alpha)}{\alpha(\exp(a/\alpha)-\exp(b/\alpha))}\left(-1 + \exp\left(\frac{(b-a)(\alpha+1)}{\alpha}\right)\right)
\end{eqnarray*}
It is now easy to see that $g \in \ol{K}'$ and that $\lim_{\alpha \rightarrow \mathbf{On}} g(\alpha,a,b) = \exp(b)-\exp(a)$, as desired. In the case where $a=0$ and $b=\omega$, we have $\int_0^\omega \exp(x) dx = \exp(\omega)-1$, which resolves the issue with the integration methods used in Conway's book~\cite{Con01} and Fornasiero's paper~\cite{For04}.
\end{example}

In real calculus, limits of Riemann sums are difficult to evaluate directly for most functions. In order to integrate such functions, the notion of primitive is used. However, we require the Fundamental Theorem of Calculus (FTC) in order to say that finding a primitive is the same as evaluating an indefinite integral.

We now state and prove the surreal analogue of the Extreme Value Theorem (EVT), which is required to prove the FTC. To prove the EVT, we need some results regarding ``strong compactness,'' which is defined in the surreal sense as follows:

\begin{defn}\label{defcompact}
Let $X \subset \mathbf{No}$. Then $X$ is \emph{strongly compact} if there exists a covering of $X$ by a proper set of subintervals open in $X$, where each subinterval has endpoints in $\mathbf{No}\cup\{\mathbf{On},\mathbf{Off}\}$, and if for every such covering there exists a finite subcovering.
\end{defn}

\begin{remark}
A subinterval of $X$ is by definition the intersection of a subinterval of $\mathbf{No}$ with $X$. The endpoints of a subinterval are its lower and upper bounds.

Observe that a strongly compact subset of $\mathbf{No}$ is a finite union of subintervals with endpoints in $\mathbf{No} \cup \{\mathbf{On}, \mathbf{Off}\}$.
\end{remark}

The restriction of coverings to proper sets in Definition~\ref{defcompact} is crucial, for otherwise the next lemma would not hold:

\begin{proposition}[From Fornasiero's paper~\cite{For04}] \label{forn'slemma}
$\mathbf{No}$ is strongly compact.
\end{proposition}

\begin{lemma} \label{ips-lem-2}
Let $[a,b] \subset \mathbf{No}$ be any closed subinterval with $a,b \in \mathbf{No} \cup \{\mathbf{On}, \mathbf{Off}\}$ (not necessarily bounded). Then $[a,b]$ is strongly compact.
\end{lemma}
\begin{proof}
Let $\mathcal{A} = \{A_\alpha\}$ be a proper set of subintervals open in $[a,b]$ such that $\mathcal{A}$ is a covering of $[a,b]$ (notice that such a covering exists because $a,b \in \mathbf{No} \cup \{\mathbf{On}, \mathbf{Off}\}$). For each $\alpha$, we can find subinterval $B_\alpha$ open in $\mathbf{No}$ and with endpoints in $\mathbf{No} \cup \{\mathbf{On}, \mathbf{Off}\}$ such that $A_\alpha = [a,b] \cap B_\alpha$. Now $\bigcup_\alpha B_\alpha = (c,d)$ for some $c,d \in \mathbf{No}^{\mathfrak{D}}$. Then if $c' \in (c,a] \cap (\mathbf{No} \cup \{\mathbf{On}, \mathbf{Off}\})$, $d' \in [b,d) \cap (\mathbf{No} \cup \{\mathbf{On}, \mathbf{Off}\})$, and $\mathcal{B} = \{B_\alpha\}$, $$\mathcal{C} = \mathcal{B} \cup \{(\mathbf{Off},c') \cup (d', \mathbf{On})\}$$ is an open covering of $\mathbf{No}$ by a proper set of subintervals whose endpoints are in $\mathbf{No}\cup\{\mathbf{On},\mathbf{Off}\}$. By Lemma~\ref{forn'slemma}, $\mathcal{C}$ has a finite subcovering $\mathcal{C}'$. Then the covering $\{C \cap [a,b] : C \in \mathcal{C}'\}$ (after removing the empty set if it appears) is a finite subcovering of $\mathcal{A}$ that covers $[a,b]$. So, the closed interval $[a,b]$ is strongly compact.
\end{proof}

\begin{remark}
It is not necessarily true that if $X$ is strongly compact and $X' \subset X$ is closed in $X$, then $X'$ is strongly compact. For example, the interval $(\infty, \mathbf{On})$ is closed because its complement $(\mathbf{Off}, \infty)$ is open, but $(\infty, \mathbf{On})$ does not have a covering by a proper set of open subintervals with endpoints in $\mathbf{No} \cup \{\mathbf{On}, \mathbf{Off}\}$.
\end{remark}

In order for the EVT to hold, we need to restrict the functions we consider to those that are \emph{strongly continuous}:

\begin{defn}\label{defstrongcont}
Let $A \subset \mathbf{No}$, and let $f : A \to \mathbf{No}$ be a function. Then $f$ is \emph{strongly continuous} on $A$ if $f$ is continuous on $A$ and if for every strongly compact bounded $A' \subset A$, there exists a covering of $f(A')$ by a proper set of subintervals open in $f(A')$ and with endpoints in $\mathbf{No} \cup \{\mathbf{On}, \mathbf{Off}\}$.
\end{defn}

\begin{example}\label{strongcontexamp}
Without the added condition of Definition~\ref{defstrongcont}, the EVT would not hold. Indeed, consider the function $f$ defined as follows. Let $g$ be the gap defined by $$g = \sum_{\alpha \in \mathbf{On}} (-1)^\alpha \cdot \omega^{-\alpha},$$
where $(-1)^\alpha$ is taken to be $1$ when $\alpha$ is a limit-ordinal. Let $g_\beta$ be the $\beta^{\mathrm{th}}$ partial sum in the normal form of $g$. Then let $f(g_\beta)$ be given as follows:
$$f(g_\beta) = \sum_{\alpha \in \mathbf{On}_{\leq \beta}} \omega^{-\alpha}.$$
For all $x > g_0$, let $f(x) = f(g_0)$ and for all $x < g_1$, let $f(x) = f(g_1)$. For all other $x$, let $f$ be the piecewise-linear function that joins the values of $f$ at the points $g_\beta$. Then notice that $f$ is strictly increasing on $(g_1,g)$ and is strictly decreasing on $(g,g_0)$, so $f$ fails to attain a maximum value on $[g_1,g_0]$, even though $f$ is continuous; i.e. $f$ violates the EVT. But we also claim that $f$ is not strongly continuous. Indeed, notice that $f([g_1, g_0]) = [1,h)$, where $h = \sum_{\alpha \in \mathbf{On}} \omega^{-\alpha}$, and there does not exist a cover of $[1,h)$ by a proper set of subintervals open in $[1,h)$ and with endpoints in $\mathbf{No} \cup \{\mathbf{On}, \mathbf{Off}\}$. This example justifies the terminology ``strongly continuous.''
\end{example}

\begin{lemma} \label{ips-lem-3}
Let $A \subset \mathbf{No}$ be a strongly compact and bounded, and let $f: A \rightarrow \mathbf{No}$ be strongly continuous. Then $f(A)$ is strongly compact.
\end{lemma}
\begin{proof}
Let $\mathcal{A} = \{A_\alpha\}$ be a covering of $f(A)$ by a proper set of subintervals open in $A$ and with endpoints in $\mathbf{No} \cup \{\mathbf{On}, \mathbf{Off}\}$ (such a covering exists because $A$ is strongly compact and bounded and $f$ is strongly continuous). For each $\alpha$, we can find subinterval $B_\alpha$ open in $\mathbf{No}$ and with endpoints in $\mathbf{No} \cup \{\mathbf{On}, \mathbf{Off}\}$ such that $A_\alpha = f(A) \cap B_\alpha$. Then let $\mathcal{B} = \{B_\alpha\}$. Note that for all $B\in\mathcal{B}$, $f^{-1}(B)$ is a union over a proper set of subintervals open in $\mathbf{No}$ and with endpoints in $\mathbf{No} \cup \{\mathbf{On}, \mathbf{Off}\}$ by Definition~\ref{defcont} because $f$ is continuous. In this regard, denote $f^{-1}(B) = \bigcup_{\alpha \in S_B} C_{B,\alpha}$, where $C_{B,\alpha}$ is a subinterval open in $\mathbf{No}$ with endpoints in $\mathbf{No} \cup \{\mathbf{On}, \mathbf{Off}\}$ for each $\alpha$ and the union is taken over a proper set $S_B$. Then $\mathcal{C} = \{f^{-1}(C_{B,\alpha})\cap A : B\in\mathcal{B}, \alpha \in S_B\}$ is a covering of $A$ with subintervals open in $A$ and with endpoints in $\mathbf{No} \cup \{\mathbf{On}, \mathbf{Off}\}$. Since $A$ is strongly compact, $\mathcal{C}$ has a finite subcovering, say $\{f^{-1}(C_i)\cap A : i \in \mathbb{N}_{\leq n}\}$ for some $n \in \mathbb{N}$. For each $i$, there exists $B_i \in \mathcal{B}$ such that $C_i \subset B_i$. Then $\{B_i \cap f(A): i \in \mathbb{N}_{\leq n}\}$ is a finite subcovering of $\mathcal{A}$ that covers $f(A)$. So, $f(A)$ is strongly compact.
\end{proof}

\begin{theorem}[EVT] \label{ips-cor-1}
Let $A \subset \mathbf{No}$ be a strongly compact and bounded, and let $f: A \rightarrow \mathbf{No}$ be strongly continuous and bounded. Then there exists $c,d \in A$ such that $f(c) \leq f(x) \leq f(d)$ for all $x \in A$.
\end{theorem}
\begin{proof}
Because $A$ is a strongly compact and bounded and $f$ is strongly continuous, $f(A)$ is strongly compact by Lemma~\ref{ips-lem-3}. Suppose that $\sup(f(A))$ is a gap. Since $f$ is bounded, we have that $\sup(f(A)) < \mathbf{On}$. Then $f(A)$ cannot be covered by a finite number of open subintervals with endpoints in $\mathbf{No} \cup \{\mathbf{On}, \mathbf{Off}\}$, so $f(A)$ is not strongly compact, which is a contradiction. Thus, $\sup(f(A)) = d \in \mathbf{No}$, and by a similar argument, $\inf(f(A)) = c \in \mathbf{No}$. Now suppose that $f$ does not attain an absolute maximum value on $A$. Then, $f(A)$ is either $[c,d)$ or $(c, d)$ and is therefore not strongly compact, which is again a contradiction. So, $f$ attains an absolute maximum value $d$ on $A$, and by a similar argument, $f$ attains an absolute minimum value $c$ on $A$.
\end{proof}

 To prove the FTC on $\mathbf{No}$, we also need (at least) a characterization of the definite integral of a function $f$ on the interval $[a,b]$ that works for all strongly continuous bounded functions $f$. (Our extrapolative method of evaluating Riemann sums only works for functions that satisfy the conditions of Definition~\ref{ips-def-4}.) We present our characterization of integration as follows:

\begin{defn}\label{intdef}
The definite integral of a strongly continuous bounded function $f$ on an interval $[a,b]$ with $a,b \in \mathbf{No}$ is a function $T(a,b)$ that satisfies the following three properties: (1) If for all $x \in [a,b]$ we have that $f(x) = c$ for some $c \in \mathbf{No}$, $T(a,b) = c(b-a)$. (2) If $m$ is the absolute minimum value of $f$ on $[a,b]$ and $M$ is the absolute maximum value of $f$ on $[a,b]$ ($m,M$ exist by Lemma~\ref{ips-lem-2} and Theorem~\ref{ips-cor-1}), then $m(b-a) \leq T(a,b) \leq M(b-a)$; and (3) for any number $c \in [a,b]$, $T(a,c) + T(c,b) = T(a,b)$.
\end{defn}

\begin{remark}
Note that in Definition~\ref{intdef}, we do not characterize $T(a,b)$ completely; we merely specify the properties that a definite integral must have in order for the FTC to be true. In fact, $T(a,b)$ can be any function having these properties, and the FTC will still hold. Observe that if we only consider functions that satisfy the requirements of Definition~\ref{ips-def-4}, our extrapolative notion of Riemann sums does satisfy Definition~\ref{intdef}.
\end{remark}

We are now ready to state and prove the FTC on $\mathbf{No}$.

\begin{theorem}[FTC]\label{ftc}
If $f$ is strongly continuous and bounded on $[a,b] \subset \mathbf{No}$, then the function $g$ defined for all $x \in [a,b]$ by
$g(x) = \int_a^x f(t) dt$
is weakly continuous on $[a,b]$ and satisfies $g'(x) = f(x)$ for all $x \in (a,b)$.
\end{theorem}
\begin{proof}
The standard proof of the FTC from real analysis (see Spivak~\cite{Spi08} for this proof) works for surreals because: (1) We can find derivatives of functions using Definition~\ref{defs-def-6}; (2) the Extreme Value Theorem holds on $\mathbf{No}$; and (3) we have characterized integration in Definition~\ref{intdef}. We outline the proof as follows.

Pick $x, x+h \in (a,b)$. Then by Definition~\ref{intdef}, $g(x+h) - g(x) = \int_{x}^{x+h} f(t) dt$, so as long as $h \neq 0$,
\begin{equation} \label{ftcproof1}
\frac{g(x+h)-g(x)}{h} = \frac{1}{h}\int_{x}^{x+h}f(t)dt.
\end{equation}
Suppose $h> 0$. We know that $f$ is continuous on $[x,x+h]$, so by the EVT, there exist $c,d \in [x,x+h]$ such that $f(c)$ is the absolute minimum value of $f$ on $[x,x+h]$ and $f(d)$ is the absolute maximum value of $f$ on $[x,x+h]$. By Definition~\ref{intdef}, we know that $h \cdot f(c) \leq \int_x^{x+h}f(t)dt \leq h \cdot f(d)$, so substituting the result of (\ref{ftcproof1}) we have:
\begin{equation}
f(c) \leq \frac{g(x+h)-g(x)}{h} \leq f(d),
\end{equation}
which holds when $h<0$ too, but the argument is similar so we omit it. If we let $h \rightarrow 0$, it is clear by the Squeeze Theorem (which follows from Definition~\ref{defs-def-6} and Theorem~\ref{diff-thm-1}) that $g'(x) = f(x)$.
\end{proof}

The FTC tells us that given a suitable definition of integral, the integral of a function is also its primitive. It is therefore natural to wonder whether one can define integration on surreals by antidifferentiation. One issue with relying solely upon antidifferentiation to evaluate integrals is that surreal functions do not have unique primitives, even up to additive constant. For example, in the case of the function $f(x) = 1$, there are many possible strongly continuous primitives in addition to $F(x) = x$, including
\begin{equation} \label{piecefunc}
F(x)  =  \begin{cases} x & x < \infty\textcolor[rgb]{1.00,1.00,1.00}{,}\\ x-1 & x > \infty.     \end{cases}
\end{equation}
The piecewise function $F$ above is strongly continuous because both $x$ and $x-1$ approach $\infty$ as $x$ approaches $\infty$, which would not be the case if $x-1$ were replaced by, say, $x-\omega$ in the definition of $F$.

We might want the primitive of a genetic function to be genetic. One reason is that if we have a genetic definition of integration (i.e. an integral that yields a genetic formula when given the genetic formula of a function), then the integral of a genetic function will be genetic by construction. By the FTC, the integral of a function is also a primitive, and so we would want the primitive of a genetic function to be genetic. In this regard, we make the following conjecture:

\begin{conjecture}\label{primconj}
Let $f : A \to \mathbf{No}$ be a genetic function defined on a locally open subinterval $A \subset \mathbf{No}$. If there exists a genetic function $F : A  \to \mathbf{No}$ such that $F'(x) = f(x)$ for all $x \in A$, then $F$ is unique up to additive constant.
\end{conjecture}

If Conjecture~\ref{primconj} is true, we would be able to integrate a genetic function $f$ using the method of antidifferentiation, for it would just suffice to find a genetic primitive $F$. Note that this method of integration would work for all surreal functions which have known genetic primitives, not just functions for which we can evaluate limits of Riemann sums using the method of extrapolation in combination with Definition~\ref{defs-def-5}.

\section{Open Questions} \label{conc}

\noindent Several open questions remain. In order to complete the analogy between real and surreal functions, consistent genetic formulae of other transcendental functions, such as sine, along with their necessary properties remain to be found. The most significant open problem that remains in surreal analysis is finding a genetic formula or a Dedekind representation for the definite integral of a function. Such a definition of integration must also satisfy the requirements of Definition~\ref{intdef}.

Two other aspects of real analysis that remain incomplete for surreals are series and differential equations. A method of evaluating series in greater generality remains to be developed, one that does not depend on the form of the $n^\mathrm{th}$ partial sum. In addition, using such a method to evaluate power series should allow basic properties, such as $f(x)=(\mathrm{power~series~of~}f(x))$ on its region of convergence, to hold. To extend surreal analysis even further, it is necessary to investigate functions of multiple variables as well as more general versions of the results presented in this paper. For example, a future study could consider proving a surreal version of Stokes' Theorem as a generalization of the FTC once a consistent theory of surreal differential forms has been developed.

A comprehensive study of differential equations remains to be performed, and as part of such a study, many questions should be answered. The following are two possible questions about the behavior of an analytic function $f$ under basic calculus operations that such a study should answer: (1) What is $\frac{d^\alpha}{{dx}^\alpha}f(x)$ for any $\alpha \in \mathbf{On}$?; and (2) Does the function that results from integrating $f$ some ordinal $\alpha$ number of times and then differentiating $\alpha$ times equal $f$ for all $\alpha \in \mathbf{On}$? Finding answers to such questions would help us understand surreal differential equations and would determine whether surreal differential equations are more general than their real analogues.

\section*{Acknowledgements}
We thank Ovidiu Costin and Antongiulio Fornasiero for pointing us to important resources and clarifying their own work on surreal numbers. We thank Philip Ehrlich for his insightful comments on the paper and for pointing us to the work of Sikorski~\cite{Sik48}. We also thank Daniel Kane, Jacob Lurie, Lyle Ramshaw, Rick Sommer, and the anonymous referees for their useful suggestions regarding the content and the presentation of this paper. A large portion of this research was conducted while the second author was a student at The Harker School, San Jose, CA 95129, USA.

\bibliographystyle{alpha}
\bibliography{bibfilev2}
\end{document}